\documentclass[a4paper,11pt,english,twoside]{article}

\newif\ifRR
\ifRR
  \usepackage{RR}
\fi

\usepackage[latin1]{inputenc}
\usepackage[T1]{fontenc}
\usepackage{amsmath,amssymb,graphicx,amsfonts,amsthm}
\usepackage{enumitem}
\usepackage{lmodern}
\usepackage{microtype}
\usepackage{bm}
\usepackage{dsfont}
\usepackage[numbers]{natbib}
\usepackage{nicefrac}
\usepackage{tikz}
\usepackage{booktabs}
\usepackage[autolanguage]{numprint}
\usepackage{multirow}
\usepackage{hyperref}
\hypersetup{colorlinks=false,pdfborder=0 0 0.2}
\providecommand{\texorpdfstring}[2]{#1}

\newtheorem{Rq}{Remark}
\newtheorem{Def}{Definition}
\newtheorem{proposition}{Proposition}
\newtheorem{lemma}[proposition]{Lemma}

\newcommand{\F}{\mathbb{F}}
\newcommand{\V}{\mathbb{V}}
\newcommand{\E}{\mathbb{E}}

\renewcommand{\P}{\mathbb{P}}
\newcommand{\G}{\mathcal{G}}
\newcommand{\N}{\mathcal{N}}

\newcommand{\X}{\mathcal{X}}
\newcommand{\GammaD}{\Gamma^\mathcal{L}}
\newcommand{\GammaP}{\Gamma^\mathcal{B}}
\newcommand{\XP}{X^\mathcal{B}}
\newcommand{\FP}{F^\mathcal{B}}

\newcommand{\1}{\mathds{1}}
\newcommand{\ind}[1]{\1_{\{#1\}}}

\newcommand{\x}{\>x}
\newcommand{\n}{^{\scriptscriptstyle (n)}}

\newcommand{\Psigma}{p_{\bm{\sigma}}}

\newcommand{\hP}{\mathcal{P}}
\def\egaldef{\stackrel{\mbox{\tiny def}}{=}}
\def\DD{\displaystyle}

\DeclareMathOperator*{\cov}{\mathrm{cov}}
\DeclareMathOperator*{\argmax}{\mathrm{argmax}}

\ifRR
\RRNo{8435}
\RRdate{décembre 2013}
\let\title=\RRetitle
\let\author=\RRauthor
\titlehead{Latent Binary Variables for Online Reconstruction of
  Large Scale Systems}
\RRtitle{Reconstruction en ligne dans des grands systèmes grâce à des
  variables binaires latentes}
\authorhead{V. Martin, J.-M. Lasgouttes \& C. Furtlehner}

\RRabstract{We propose a probabilistic graphical model realizing a
  minimal encoding of real variables dependencies based on possibly
  incomplete observation and an empirical cumulative distribution
  function per variable. The target application is a large scale
  partially observed system, like e.g.\ a traffic network, where a
  small proportion of real valued variables are observed, and the
  other variables have to be predicted. Our design objective is
  therefore to have good scalability in a real-time setting. Instead
  of attempting to encode the dependencies of the system directly in
  the description space, we propose a way to encode them in a latent
  space of binary variables, reflecting a rough perception of the
  observable (congested/non-congested for a traffic road). The method
  relies in part on message passing algorithms, i.e.\ belief
  propagation, but the core of the work concerns the definition of
  meaningful latent variables associated to the variables of interest
  and their pairwise dependencies. Numerical experiments demonstrate
  the applicability of the method in practice.}

\RRresume{On propose un modèle graphique probabiliste réalisant un
  encodage parcimonieux des dépendances d'un jeu de variables réelles
  à partir de corrélations possiblement incomplètes et de la fonction
  de répartition de chaque variable. L'application sous-jacente est un
  grand système partiellement observé, comme par exemple un réseau
  routier, où une petite proportion de variables réelles sont
  observées et les autres doivent être prédites. Il s'agit donc de
  trouver un algorithme permettant de travailler en temps réel sur des
  grands systèmes. Au lieu d'essayer d'encoder directement les
  dépendances du système, on propose une manière de le faire dans un
  espace de variables binaires latentes qui reflètent une vision
  sommaire des variables (par exemple embouteillé/fluide pour des
  routes). La méthode s'appuie en partie sur des algorithmes de
  passage de message comme «~belief propagation~», mais la difficulté
  est surtout de définir des variables binaires cohérentes avec celles
  de départ, ainsi que leurs corrélations. L'applicabilité de la
  méthode en pratique est démontrée par des expérimentation
  numériques.}

\RRkeyword{latent variables, Markov random field, belief propagation,
inference, soft constraints}
\RRmotcle{variables latentes, champ markovien aléatoire, propagation
  de croyances, inférence, contraintes bayésiennes}
\RRprojets{Imara and TAO}
\RCParis
\else
\date{}

\fi

\title{Using Latent Binary Variables for Online Reconstruction of
  Large Scale Systems}

\author{Victorin Martin%
  \thanks{Mines-Paristech, Paris, France, e-mail: \texttt{victorin.martin@mines-paristech.fr}}
    \and Jean-Marc Lasgouttes\thanks{Inria, Imara Project-Team,
      e-mail: \texttt{jean-marc.lasgouttes@inria.fr}}
  \and Cyril Furtlehner\thanks{Inria, TAO Project-Team, e-mail: \texttt{cyril.furtlehner@inria.fr}}
  }
\begin{document}

\ifRR
\makeRR
\else
\maketitle

\begin{abstract}
We propose a probabilistic graphical model realizing a minimal encoding of real
variables dependencies based on possibly incomplete observation  and an empirical
cumulative distribution function per variable. The target application is a large scale partially observed system, like
e.g.\ a traffic network, where a small proportion of real valued variables are
observed, and the other variables have to be predicted. Our design
objective is therefore to have good scalability in a  real-time
setting. Instead of
attempting to encode the dependencies of the system directly in the
description space, we propose a way to encode them in a latent space of binary
variables, reflecting a rough perception of the observable (congested/non-congested
for a traffic road). The method relies in part on message passing algorithms, i.e.\
belief propagation, but the core of the work concerns the definition of
meaningful latent variables associated to the variables of
interest and their pairwise dependencies. Numerical experiments demonstrate
the applicability of the method in practice.
\end{abstract}

\fi

\section{Introduction}\label{sec:intro}
Predicting behavior of large scale complex stochastic systems
is a relevant question in many different situations where
a (communication, energy, transportation, social, economic\ldots) network
evolves for instance with respect to some random demand and limited supply.
This remains to a large extent an open and considerable problem, especially for partially observed
systems with strong correlations (see e.g.~\citet{Boyen}), though
efficient  methods, like Kalman and, by extension, particle filtering, see e.g.~\citet{Doucet}, exist,
but with limited scalability. 

In the example which motivates this work, road traffic reconstruction
from floating car data, the system is partially observed and the goal
is to predict the complete state of the traffic network, which is
represented as a high-dimensional real valued vector of travel times or
alternatively speeds or densities. State of the art methods in this field
exploit both temporal and spatial correlations with multivariate
regression~(\citet{MiWy}) with rather restrictive linear hypothesis on
the interactions. Here, we explore a different route for the encoding
of spatial and potentially temporal dependencies, which we believe can
simplify both the model calibration and the data reconstruction tasks
for large scale systems.

The classical way to obtain data on a road traffic network is to
install fixed sensors, such as magnetic loops. However, this is
adapted to highways and arterial roads, but not to a whole urban
network which typically scales up to $10^5$ segments. As part of the
Field Operational Test PUMAS~\cite{PUMAS} in Rouen (Normandy), we
explored the possibility to acquire data with equipped vehicles that
send geolocalized information, and to process it directly with a fast
prediction scheme (\citet{FuLaFo,ITSC10}). While offline processing of
historical data can be allowed to be time consuming, travel times
predictions must instead be available in ``real-time'', which means in
practice a few minutes. This ``real-time'' constraint implies some
design choices: firstly, the predictions need to be computed online,
even on large networks, which can be achieved using the
message-passing algorithm Belief Propagation of~\citet{Pearl};
secondly, our model shall be suitable for the use of this inference
algorithm.

Stated in a more generic form, the problem at hand is to predict, from sparse data originating from
non stationary locations, the value of the variables on the rest of the network. The set of nodes to predict 
is potentially varying from time to time because sensors are moving, like probe vehicles in the traffic context.
Since only very sparse joint observations are available, purely data driven methods
such as $k$ nearest neighbors cannot be used and one has to resort to building some
model. We study in this article the possibility of a probabilistic graphical model that avoids
modeling the underlying complexity of the physical phenomena. Building a model of
dependency between real-valued variables can be very costly both in terms of
statistics, calibration and prediction if one tries to account for the empirical
joint probability distribution for each pair of variables. 
A possible way to proceed, compatible with the use of BP, is to build a multivariate Gaussian Copula,
since then Gaussian belief propagation can be run efficiently on such models.
However, in the traffic example, the variables are endowed with a binary perception
(congested/non-congested), which make it likely that the joint distribution will be multi modal 
in general and therefore not very well suited for a Gaussian model, which admits only one reference state
associated to one single belief propagation fixed point.
What we propose instead is to abstract the binary perception as a latent state
descriptor and to exploit it by encoding the dependencies between the real state
variables at the level of the latent variables. This way
we end up with a minimal parametric method, where both the prediction and the calibration are easy
to perform and which is well adapted to multimodal distributions; each mode can be 
under certain conditions associated to a belief propagation fixed point~\cite{FuLaAu}.
Another route could be to build a general continuous copula model, compatible with  
the use of the expectation propagation (EP) algorithm of \citet{Minka}. The EP
algorithm is indeed also very simple and efficient to use, but the model selection stage might be too 
complex for large scale applications. 
It requires to choose manually the exact exponential family for
modeling the marginal and pairwise distributions, since no automatic procedures has emerged up to now.
Also, compared to traditional methods like particle filtering, we expect a better scalability.

We formalize the model as follows: the state of the system is
represented by a vector $\>X=(X_i)_{i\in\V}$ of $N$ real valued
variables, attached to nodes $i \in \V$ and taking their respective
values in the sets $\X_i\subset {\mathbb R}$. We assume that we never
observe the full vector $\>X$, but that only pairwise observations are available. For
a given set $\E \subset \V^2$ of pair of variables, each pair $(X_i,X_j)$ such as
$(i,j) \in \E$ is observed $N_{ij}$ times, all observations being independent. These
pair samples are stored in the vector $\x$, which contains all the vectors
$\x^k$:
\begin{equation}\label{eq:xobs}
(X_i,X_j) = (x_i^k,x_j^k) \text{ for } k \in \{1,\ldots, N_{ij}\}.
\end{equation}

\bigskip

The model goes as follows: to each variable $X_i$ is attached a binary latent
variable $\sigma_i$ and the variables $X_i$ are assumed to be independent,
conditionally to the latent state $\bm\sigma$. This is a strong
assumption, generally false, but we shall see in this paper that it can provide an
efficient model for the prediction task. We wish to stress that it will be necessary
to \emph{construct} these latent variables and multiple choices could be meaningful.
Somehow this is a choice of feature functions from $\X_i$ to $\{0,1\}$.
The problem at stake is not to infer the states of a hidden Markov
model from noisy observations: the only variables of interest are the $X_i$'s. To be
able to infer the behavior of these variables, given a partial observation
of the system, we use a pairwise Markov Random Field (MRF) for the
binary variables $\sigma_i$, i.e.\ an Ising model in statistical physics parlance
(\citet{Baxter}). The joint measure for the variables $\textbf{X}$ and
$\bm{\sigma}$
factorizes as (Figure~\ref{fig:genmodel}):
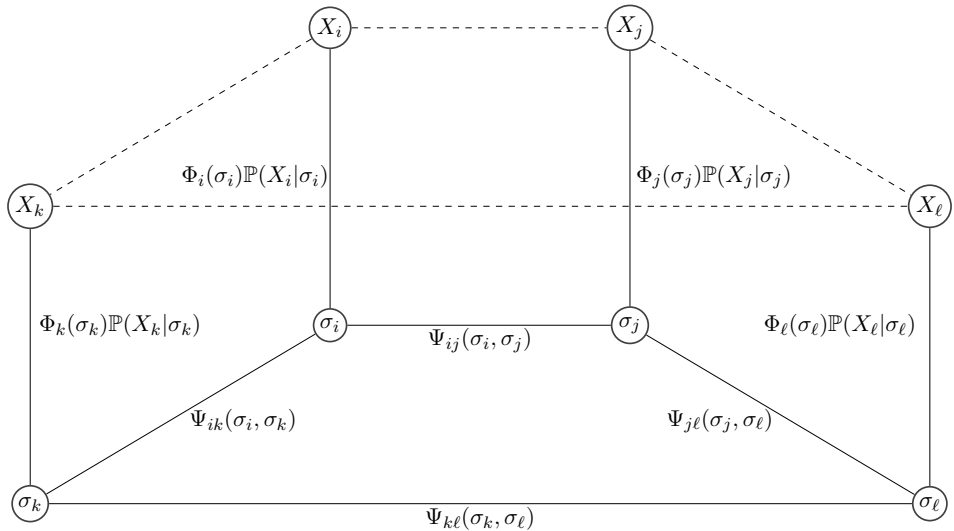
\begin{figure}
\resizebox{0.99\textwidth}{!}{\begin{tikzpicture}[scale=1]
 
 \tikzstyle{varp}=[circle,thick,draw=black!75,inner sep=1.5pt,minimum
size=1pt,font=\normalsize]
 \tikzstyle{no}=[circle,thick,draw=black!75,inner sep=1.5pt,minimum
size=1pt,font=\normalsize]
\begin{scope}
 
  \node [varp] (si) at (0,0) {$\sigma_i$};
  \node [varp] (sj) at (5,0){$\sigma_j$};
  \node [varp] (sk) at (-5,-3) {$\sigma_k$};
  \node [varp] (sl) at (10,-3) {$\sigma_\ell$};
  
  \draw[] (si) -- (sj);
  \node[color=black] at (2.5,-0.25) {$\Psi_{ij}(\sigma_i,\sigma_j)$};
  \draw (si) -- (sk);
  \node[color=black] at (-1.45,-1.6) {$\Psi_{ik}(\sigma_i,\sigma_k)$};
  \draw (sj) -- (sl);
  \node[color=black] at (6.5,-1.6) {$\Psi_{j\ell}(\sigma_j,\sigma_\ell)$};
  \draw (sk) -- (sl);
  \node[color=black] at (2.5,-3.25) {$\Psi_{k\ell}(\sigma_k,\sigma_\ell)$};

  \node [varp] (Xi) at (0,5) {$X_i$};
  \node [varp] (Xj) at (5,5){$X_j$}; 
  \node [varp] (Xk) at (-5,2) {$X_k$};
  \node [varp] (Xl) at (10,2) {$X_\ell$};

  \draw[dashed] (Xi) -- (Xj);
  \draw[dashed] (Xi) -- (Xk);
  \draw[dashed] (Xj) -- (Xl);
  \draw[dashed] (Xk) -- (Xl);

  \draw (si) -- (Xi);
  \node[color=black] at (-1.25,2.5) {$\Phi_i(\sigma_i) \P(X_i|\sigma_i)$};
  \draw (sj) -- (Xj);
  \node[color=black] at (6.4,2.5) {$\Phi_j(\sigma_j) \P(X_j|\sigma_j)$};
  \draw (sk) -- (Xk);
  \node[color=black] at (8.5,0) {$\Phi_\ell(\sigma_\ell) \P(X_\ell|\sigma_\ell)$};
  \draw (sl) -- (Xl);
  \node[color=black] at (-3.5,0) {$\Phi_k(\sigma_k) \P(X_k|\sigma_k)$};

\end{scope} 
  
\end{tikzpicture}}
\caption{Markov random field $(\mathbf{X},\bm{\sigma})$ for 
$\V =\{i,j,k,\ell\}$. The true model of the vector $\mathbf{X}$ (dashed
lines) is approximated through the latent binary variables $\bm{\sigma}$ (plain
lines).\label{fig:genmodel}}
\end{figure}
\begin{equation*}
 \P(\textbf{X}\leq \x,\bm{\sigma}=\mathbf{s}) =\P(\bm{\sigma}=\mathbf{s})
 \prod_{i \in \V}\P(X_i\leq x_i|\sigma_i=s_i),
\end{equation*}
\begin{equation*}
 \P(\bm{\sigma}=\mathbf{s}) = \frac{1}{Z}\prod_{(i,j) \in \E} \psi_{ij}(s_i,s_j)
 \prod_{i
\in \V} \phi_i(s_i),
\end{equation*}
with $Z$ a constant ensuring that $\P$ sums up to $1$. Of course it will not
be possible to model in a precise way the joint distribution of any random vector
$\>X$ through our latent Ising model. The task assign to the model is actually less
ambitious: we wish to make predictions about this random vector $\>X$. The problem
we are trying to solve is simply a regression on the variables $X_i$,
which is very different in nature from modeling the distribution $\hP$.
Note that with observations \eqref{eq:xobs} Jaynes' maximum entropy criterion leads us 
to a pairwise interaction model which is compatible with our choice.

Based on these assumptions, we try to answer three main questions:
\begin{itemize}

  \item[(i)] How to define the latent variable $\sigma_i$ and how to relate it to its real
  valued variable $X_i$?

  \item[(ii)] How to construct the dependencies between latent
    variables $\sigma_i$ in an efficient way in terms of prediction
    performance?

  \item[(iii)] How can partial observations be inserted into the model
    to perform the predictions of the unobserved variables?
\end{itemize}
These three questions are of course highly interdependent. Once the
model has been built, exact procedures to infer the behavior of the
$X_i$'s generally face an exponential complexity, and one has to
resort to an approximate procedure. We rely here on Pearl's belief
propagation (BP) algorithm~\citep{Pearl} -- widely used in the Artificial 
Intelligence and Machine Learning communities \cite{Kschi,YeFrWe3} -- 
as a basic decoding tool, which in turn will influence the MRF
structure. While this algorithm is well defined for real-valued
variables in the case of a Gaussian vector $\textbf{X}$
(see~\citet{Bickson}), the more general case requires other
procedures, like the nonparametric BP algorithm proposed
by~\citet{Sudderth10}, which involves much more computation than the
classical BP algorithm does. We propose here a new BP-based method to
tackle this same problem while keeping computations lightweight. The
BP algorithm will be precisely defined in Section~\ref{sec:bpalgo}.

The paper is organized as follows: Section~\ref{sec:real2binary} is
devoted to answering question (i), by finding a relevant mapping of an
observation $X=x$ to the parameter of a Bernoulli variable $\sigma$. 
As we shall, see this is equivalent to the definition
of a feature function.
Section~\ref{sec:buildmodel} focuses on question (ii) concerning the
optimal encoding in the latent space of the dependency between the
$X_i$'s. In Section~\ref{sec:bpalgo}, we construct a variant of BP
named ``mirror BP'' that imposes belief values of $\sigma_i$ when
$X_i$ is observed; this addresses question (iii).
Some experimental results of these methods are presented in
Section~\ref{sec:simul}.

\section{Latent variables definition}\label{sec:real2binary}

Let $X$ be a real-valued random variable with cumulative distribution function (cdf) 
 $F(x)\egaldef\P(X\leq x)$. We focus in this section on a way to relate an
observation $X=x$ to a latent binary variable $\sigma$. In the following we will
call ``$\sigma$-parameter'' the value $\P(\sigma=1)$.

\subsection{A stochastically ordered mixture}\label{ssec:randthreshold}
A simple way to relate an observation $X=x$ to the latent variable $\sigma$ is
through a mapping $\Lambda$ such that $\Lambda(x)$ is the $\sigma$-parameter.
The mapping $\Lambda$ will be referred to as the encoding function and can depend on
the cdf $F$. $\sigma$ being a latent variable, it will not be directly observed, but
conditionally to an observation $X=x$, we define its distribution as:
\begin{equation}
\label{eq:deflambda}
  \P(\sigma=1|X=x) \egaldef \Lambda(x).
\end{equation}

For simplicity, we assume that $\Lambda$ is continuous on right, limited on left
(\emph{corlol}) and increasing. Note that the condition ``$\Lambda$
is increasing'' is equivalent to have a monotonic mapping $\Lambda$ since choosing
the mapping $1-\Lambda$ simply inverts the states $0$ and $1$ of the variable
$\sigma$. Moreover $\Lambda$ shall increases from $0$ to $1$, without requiring that
$\Lambda(\X) = [0,1]$, since $\Lambda$ can be discontinuous. This constraint is
expressed as the following: 
\begin{equation}
 \label{eq:maxspread}
 \int_\X d\Lambda(X) = 1\quad  \text{and} \quad \inf_{x\in\X} \Lambda(x) = 0.
\end{equation}

Let us emphasize again that $\sigma$ is not just an unobserved latent
random variable which estimation is required. It is a feature that we define
in order to tackle the inference on $\>X$. This encoding is part of
the following global scheme
\begin{equation}
\label{eq:scheme}
 \begin{matrix}
  X_i = x_i \in \X_i &\overset{\Lambda_i}{\longrightarrow} &
\P(\sigma_i=1|X_i=x_i)\in \Lambda_i(\X_i) \\
  & & \Bigm\downarrow \text{mBP} \\
X_j=x_j \in \X_j & \overset{\Gamma_j}{\longleftarrow} &
b(\sigma_j=1)
\in [0,1] 
 \end{matrix}
\end{equation}
which is as follows:
\begin{itemize}
 \item observations of variables $X_i$ are encoded through the distribution of a
       latent binary random variable $\sigma_i$ using the encoding function $\Lambda_i$,
 \item a marginalisation procedure is then performed on these
       latent variables $\bm{\sigma}$, in a way that will be described
       later,
 \item and finally the distributions of variables $\sigma_j$ allows in turn to make predictions
       about the other real variables $X_j$.
\end{itemize}

This scheme requires that we associate to $\Lambda$ an ``inverse'' mapping $\Gamma:
[0,1] \mapsto \X$. Since $\Lambda$ can be non invertible, the decoding function
$\Gamma$ cannot always be the inverse mapping $\Lambda^{-1}$.
We will return to the choice of the function $\Gamma$ in Section~\ref{ssec:decod}.

To understand the interaction between $\sigma$ and $X$, let us define the conditional cdf's:
\begin{align*}
 F^0(x) &\egaldef \P(X\leq x|\sigma=0), \\
 F^1(x) &\egaldef \P(X\leq x|\sigma=1).
\end{align*}
Bayes' theorem allows us to write
\begin{equation*}
  \P(\sigma = 1|X=x) = \P(\sigma=1)\frac{dF^1}{dF}(x),
\end{equation*}
and thus
\begin{equation}\label{eq:dF1}
 dF^1(x) = \frac{\Lambda(x) }{\P(\sigma=1)}dF(x).
\end{equation}
Summing over the values of $\sigma$ imposes 
\begin{equation}
\label{eq:normcons}
 F(x) = \P(\sigma=1)F^1(x) + \P(\sigma=0)F^0(x),
\end{equation}
and the other conditional cdf follows
\begin{equation}\label{eq:dF0}
 dF^0(x) = \frac{1-\Lambda(x) }{\P(\sigma=0)}dF(x).
\end{equation}
The choice of this class of \emph{corlol} increasing functions has a simple
stochastic interpretation given in the following proposition. 
\begin{proposition}
 The choice of an increasing encoding function $\Lambda$ yields a  
separation of the random variable $X$ into a mixture of two stochastically ordered
variables $X^0$ and $X^1$ with distributions $dF^0$ and $dF^1$. Indeed, we then have
\begin{equation*}
 X \sim \ind{\sigma=0}X^0 + \ind{\sigma=1}X^1,
\end{equation*}
where $\sim$ is the equality in term of probability distribution. The
stochastic ordering is the following
\begin{equation*}
 X^0 \preceq X \preceq X^1.
\end{equation*}
\end{proposition}
\begin{proof}
It is sufficient (and necessary) to prove that 
\begin{equation*}
\forall x\in\X,\quad F^1(x) \leq F(x) \leq F^0(x).
\end{equation*}
Consider first the left inequality ($F^1\leq F$); If $x \in \X$ is such that
$\Lambda(x) \leq \P(\sigma=1)$, then we have:
\begin{align*}
 F^1(x) &= \int_{-\infty}^x dF^1(y) = \int_{-\infty}^x \frac{\Lambda(y)
}{\P(\sigma=1)}dF(y),\\
 &\leq \int_{-\infty}^x dF(y) = F(x),
\end{align*}
because, $\Lambda$ being increasing, $\Lambda(y) \leq \P(\sigma=1)$ for all $y \in
]-\infty,x]$. Conversely, when $\Lambda(x) \geq \P(\sigma=1)$,
\begin{align*}
 F^1(x) &= 1 - \int_{x}^{+\infty} dF^1(y) = 1 - \int_{x}^{+\infty} \frac{\Lambda(y)}
 {\P(\sigma=1)}dF(x),\\
 &\leq 1- \int_{x}^{+\infty} dF(y) = F(x),
\end{align*}
using again the fact that $\Lambda$ is increasing. The other inequality ($F \leq F^0$)
is obtained using~\eqref{eq:normcons}.
\end{proof}
\begin{Rq}
Since the encoding function $\Lambda$ is increasing from $0$ to $1$, it can be
considered as the cdf of some random variable $Y$,
\begin{equation*}
 \P(Y \leq x) \egaldef \Lambda(x),
\end{equation*}
which allows to reinterpret the previous quantities in terms of $Y$:
\begin{align*}
 \P(\sigma=1) &= \int_\X \P(\sigma=1|X=x)dF(x)\\
 &= \int_\X \Lambda(x) dF(x) = \int_\X \P(Y\leq x) dF(x)\\
 &= \vphantom{\int_\X} \P(Y \leq X),
\end{align*}
supposing that $Y$ and $X$ are independent. The variable $\sigma$ can therefore be
defined as 
\begin{equation*}
 \sigma \egaldef \ind{Y \leq X},
\end{equation*}
which means that the variable $Y$ acts as a random threshold separating $X$-values
that correspond to latent states $0$ and $1$. The stochastic ordering between
$(X|\sigma=0)$ and $(X|\sigma=1)$ then appears quite naturally. Note that this interpretation
leads to a natural extension to a larger discrete feature space for $\sigma$ simply 
using multiple thresholds.  
When $\Lambda=F$, the conditional cdf's of $X$ are:
\begin{align*}
 F^1(x) &= (F(x))^2 = \mathbb{P}(\max(X_1,X_2)\leq x), \\
 F^0(x) &= F(x)(2-F(x))  = \mathbb{P}(\min(X_1,X_2)\leq x),
\end{align*}
with $X_1$ and $X_2$ two independent copies of $X$.
\end{Rq}

\subsection{Choosing a good encoding function
\texorpdfstring{$\Lambda$}{Lambda}}\label{ssec:choix_encod}

Now that the nature of the mapping between $X$ and $\sigma$ has been described,
it remains to find an ``optimal'' encoding function. It turns out to be difficult to
find a single good criterion for this task. In this section, we therefore propose two
different approaches, based respectively on the mutual information and on the
entropy.

\paragraph{Mutual information.}
The idea here is to choose $\Lambda$ (or equivalently $\sigma$), such that the
mutual information $I(X,\sigma)$
between variables $X$ and $\sigma$ is maximized. In other words, a given information about 
one variable should lead to as much knowledge as possible on the other one.
\begin{proposition}
Let $q_X^{\numprint{0.5}}$ be the median of $X$. The encoding function 
$\Lambda_\mathrm{MI}$ which maximizes the mutual information $I(X,\sigma)$ between variables $X$
and $\sigma$ is the step function
\begin{equation*}
 \Lambda_\mathrm{MI}(x) \egaldef \1_{\{x\,\geq\, q_X^{\numprint{0.5}}\}}.
\end{equation*}
\end{proposition}
Before turning to the proof of this proposition, let us remark that this definition
of the binary variable $\sigma$ is a natural one, $\sigma$ being deterministic as a
function of $X$. However, as we shall see in Section~\ref{sec:simul}, it is usually
suboptimal for the reconstruction task.

\begin{proof}
The function to maximize is
\begin{align*}
 I(X,\sigma) &= \sum_s \int_{\X} \P(\sigma=s)\log\left(
\frac{dF^s(x)}{dF(x)}\right)dF^s(x)\\
&= H(\P(\sigma=1)) - \int_{\X} H(\Lambda(x))dF(x),
\end{align*}
where $H(p) \egaldef -p\log p -(1-p)\log(1-p)$ is the binary entropy function. Among
all random  variables $\sigma$ with entropy $H(\P(\sigma=1))$, the ones which
maximize $I(X,\sigma)$ are deterministic functions of $X$, or equivalently the ones
for which $\Lambda$ is an indicator function. Since we limit ourselves to the
\emph{corlol} class,
we get $\Lambda = \1_{[a,+\infty[}$ for some $a \in \X$. It remains to maximize
the entropy of the variable $\sigma$, which leads to $P(\sigma=1) =
\nicefrac{1}{2}$ and $a = q_X^{\numprint{0.5}}$.
\end{proof}

\paragraph{Max-entropy principle.} 
Another possibility is, in order to maximize the information contained in the latent
variable $\sigma$, to maximize the entropy of $U=\Lambda(X)$. This variable $U$ is
indeed the data that will be used to build the Ising model over the latent variables
(see Section~\ref{sec:buildmodel}). We assume here that the variable $X$ admits a
probability density function (pdf). We add to the few constraints detailed in the
previous section that $\Lambda$ is a bijection between $\X$ and $[0,1]$.  When
dealing with continuous random variables, the entropy only makes sense relatively
to some measure (see~\citet[pp. 374-375]{Jaynes2003}). Following
Jaynes'~\cite{Jaynes1968} arguments, since $U$ is the parameter of a Bernoulli
variable with both outcomes possible, having no other prior knowledge leads us to the
uniform measure as reference. 

\begin{proposition}
\label{prop:maxentcriterion} 
Let $X$ be a random variable which admits a pdf. The (increasing) invertible function
which maximizes the entropy of $U = \Lambda(X)$, taken relatively to the uniform
measure, is the cumulative distribution function $F$ of the variable $X$.
\end{proposition}
\begin{proof}
The variable $U$ with maximal entropy has a uniform pdf $h_\Lambda(u) = \1_{[0,1]}(u)$.
The encoding function $\Lambda$ such as $\Lambda(X)$ is a uniform variable on $[0,1]$
is the cdf of $X$, which concludes the proof. 
\end{proof}

\smallskip

\noindent Let us quickly sum up the choices we proposed for the encoding function:
\begin{itemize}
 \item $\Lambda_\text{MI}$ which is a deterministic encoding: $\sigma$
indicates the position of $X$ w.r.t.\ its median;
 \item $F$ the cdf of $X$, which corresponds to the less discriminating choice about
the encoded data distribution.
\end{itemize}
We will see that these two encoding functions have very distinct properties:
$\Lambda_\text{MI}$ is much more conservative than $F$ but is rather
adapted to model precisely the joint distributions. Let us remark that in the first case
$\Lambda$ is the feature function while in the second case the feature function is a 
random variable.

\subsection{Decoding function \texorpdfstring{$\Gamma$}{Gamma}}\label{ssec:decod}

Before turning to the definition of the decoding function $\Gamma$, let us focus first 
on the following simple question:
\begin{quote}
 \emph{What is the best predictor of a real-valued random variable $X$, knowing only
its distribution?}
\end{quote}
The answer will obviously depend on the loss function considered and this will in
turn influence the choice of the decoding function $\Gamma$, which purpose is to
predict the random variable $X$. Assuming a $L^r$ norm as loss function, the optimal
predictor $\hat \theta_r(X)$ is then defined as
\begin{equation*}
 \hat \theta_r(X) \egaldef \underset{c \in
\mathbb{R}}{\operatorname{argmin}}\  \E_X[|X -c|^r].
\end{equation*}
In the case $r=1$, the optimal predictor $\hat \theta_1(X)$ is simply the median of
 $X$; $r=2$ corresponds to $\hat \theta_2(X) = \E[X]$, the mean value of $X$.
In the following we call ``contextless prediction'' the $X$-prediction performed without
other information than the distribution of $X$.

\medskip

 When focusing on the definition of the ``inverse'' mapping $\Gamma$, two
natural definitions arise. When $\Lambda$ is a bijection, the simplest predictor
of $X$, given $b=\P(\sigma=1)$, is $\Lambda^{-1}(b)$. Actually, it is the
unique $X$-value such that $(\sigma|X=x)$ is distributed as $\P(\sigma=1|X=x) 
=b$ by definition \eqref{eq:deflambda} of $\Lambda$. We will denote this first
choice for the decoding function
\begin{equation*}
 \GammaD \egaldef \Lambda^{-1}.
\end{equation*}
$\GammaD$ corresponds, in some sense, to a predictor based on maximum
likelihood (ML). Indeed, suppose that the knowledge of $b=\P(\sigma=1)$ is
replaced with a sample of $M$ independent copies $s^k$ of a binary variable
distributed as $\P(\sigma|X=x)$. The ML estimate of $x$ is then
$\Lambda^{-1}(\nicefrac{\sum_k s^k}{M})$. So the choice $\Lambda^{-1}$ as decoding
function corresponds to the ML estimate from a sample with an empirical rate of
success equal to $b$.

In the more general case of an increasing \emph{corlol} encoding function, with
a Bayesian point of view, the knowledge of the $\sigma$-parameter
allows to update the distribution of $X$. Applying Jeffrey's update rule
(see \citet{Chan2005}) yields the updated cdf $\FP$
\begin{equation}
\label{eq:Xhat}
 \FP(x) = bF^1(x) + (1-b)F^0(x).
\end{equation}
Let $\XP$ be a random variable which distribution is $\FP$. The predictor 
 $\hat \theta(\XP)$ previously defined can be used irrespective of whether
$\Lambda$ is invertible or not. To refer to this second choice we will use the
notation
\begin{equation*}
 \GammaP \egaldef \hat \theta(\XP).
\end{equation*}
Note that, while it may be costly in general to compute the wanted statistic of
$\XP$, some choices of $\Lambda$ lead to explicit formulas.

\paragraph{Mutual information.} We consider here the case of the step function 
$\Lambda_\text{MI}$ as encoding function. This function is of course not invertible
and only the Bayesian decoding function $\GammaP$ can be used. 
Using \eqref{eq:dF1}--(\ref{eq:Xhat}), the cdf of $\XP$ is
\begin{equation}\label{eq:GpMI}
\FP(x) = 
\begin{cases}
 2(1-b) F(x), &\text{if } x \leq q_X^{\numprint{0.5}},\\[2mm]
 \FP(q_X^{\numprint{0.5}}) + 2b\left(F(x)-F(q_X^{\numprint{0.5}})\right),
&\text{if } x > q_X^{\numprint{0.5}}.
\end{cases}
\end{equation}
In order to compute $\hat\theta_1(\XP)$, we need to solve the equation
$\FP(x)=1/2$, leading to the decoding function
\begin{equation*}
\GammaP(b) =
\begin{cases}
 F^{-1}\left(\frac{1}{4(1-b)}\right),&\text{if } b \leq \frac{1}{2},\\[2mm]
 F^{-1}\left(\frac{4b-1}{4b}\right),& \text{if } b > \frac{1}{2}.
\end{cases}
\end{equation*}
When $F$ is not invertible, $F^{-1}$ should be understood as the pseudo-inverse of
$F$, commonly used to define quantiles:
\begin{equation*}
F^{-1}(b) \egaldef \inf_x\{x\mid F(x)\geq
b\}. 
\end{equation*}
If we choose the predictor $\hat \theta_2$ based on a $L^2$ loss function, using the 
linearity of the expectation we get
\begin{equation*}
 \GammaP(b) = \E\left[\XP\right] = b\,\E\left[X\mid \sigma=1\right] +
(1-b)
\E\left[X \mid\sigma=0\right].
\end{equation*}

\paragraph{Max entropy principle.} When one uses $\GammaD = F^{-1}$ as
decoding function, the contextless prediction, i.e.\ without any observation, is simply
 $F^{-1}\bigl(\P(\sigma=1)\bigr)$. Moreover, we know that $\P(\sigma=1) = \E[F(X)] =
\nicefrac{1}{2}$ -- provided that $X$ admits a pdf -- so the ground prediction is the
median of $X$. The choice $\Lambda=F$ and $\Gamma=F^{-1}$ is therefore optimal
w.r.t.\ a $L^1$ loss function for the prediction error.

The other choice for the decoding function is to use $\GammaP$ and to
compute, for example, the predictor $\hat\theta_1(\XP)$. Using 
\eqref{eq:dF1} -- \eqref{eq:Xhat}, we get the cdf of $\XP$
\begin{equation*}
 \FP(x) = \bigl( (2b-1)F(x) - 2(b-1)\bigr)F(x),
\end{equation*}
and the sought function is solution of the following quadratic equation
\begin{equation*}
 \left( (2b-1)F(x) - 2(b-1)\right)F(x) = \frac{1}{2},
\end{equation*}
with only one reachable root. Thus the Bayesian decoding function is
\begin{equation}\label{eq:GpS}
\GammaP(b) = F^{-1}\left(\frac{2(b-1) + \sqrt{(2b-1)^2+1}}{4b-2}
\right).
\end{equation}

Let us remark that the Bayesian decoding function $\GammaP$ is always
more conservative than $\GammaD$. Using the inverse $F^{-1}$ allows us to
make predictions spanning the whole set $\X$ of possible outcomes, which is
not the case with the Bayesian decoding function. Figure~\ref{fig:Decod}
illustrates this.

Assume that two random variables $X_1$ and $X_2$ are equal with
probability $1$. Even if we build a latent model such that $\P(\sigma_1 = \sigma_2)
= 1$, using $\GammaP$ as decoding function will never predict $X_1=X_2$.
The decoding function $\GammaP$ is in fact trying to approximate
the joint distribution of ($X_1,X_2$) and this approximation can only be very rough
when variables are strongly dependent (see Proposition~\ref{prop:capac}).
However, the choice ($F,F^{-1}$) is equivalent to performing a $X$-quantiles regression.
We will see in Section~\ref{sec:simul} that this last choice is better when
variables are strongly dependent.
\begin{figure}
 \centering\resizebox{0.8\textwidth}{!}{\input{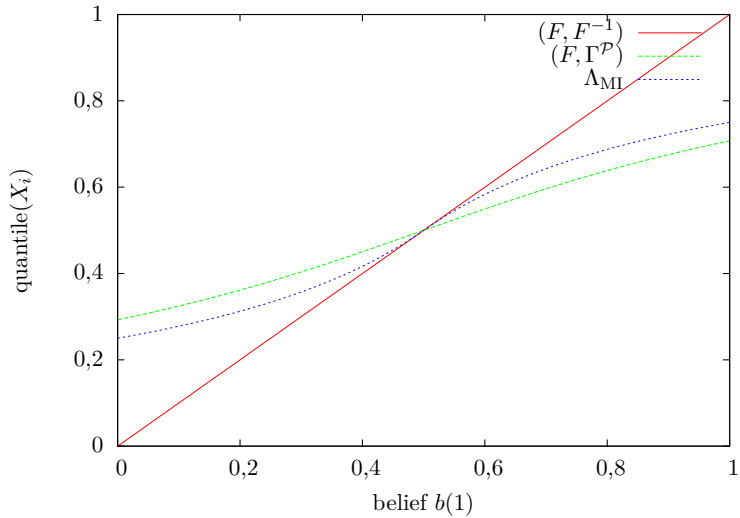}}
\caption{Prediction, expressed in quantiles, on the variable $X_i$ for a given belief
$b_i(1)$. The left and right values are ($\nicefrac{1}{4},\nicefrac{3}{4}$) for
$\Lambda_\text{MI}$ and ($1-\nicefrac{\sqrt{2}}{2}$,$\nicefrac{\sqrt{2}}{2}$) for
($F$,$\GammaP$).\label{fig:Decod}}
\end{figure}

In one wishes to choose the decoding function based on the ML estimate
$\GammaP = \Lambda^{-1}$, it is of interest to generalize the max-entropy
criterion in order to get an encoding function $\Lambda$ with an optimal contextless
prediction w.r.t.\ a specific loss function. It is in fact quite simple to solve this
problem and to obtain the sought encoding function which is based on the cdf
(\citet[chapter 5]{mythesis}). We will use here only the cdf function because we
are interested in $L^1$ error measure. Compared to a loss function based on the $L^2$
norm, it gives less weight to extreme values.

\section{Building pairwise dependencies}\label{sec:buildmodel}

It was shown in Section~\ref{sec:real2binary} how to relate the variable $X_i$ to its
latent state $\sigma_i$, by means of an encoding function $\Lambda_i$. The next question
to address is how to encode the dependencies at the latent state level and, more generally,
how to estimate the parameters of the underlying Ising model on $\bm{\sigma}$. Given
two real-valued variables $X_i$ and $X_j$, with respective cdf $F_i$ and $F_j$, and
two binary variables $\sigma_i$ and $\sigma_j$, we want to construct a pairwise model
 as described in Figure~\ref{fig:genmodel}. The probability distribution
of the vector $(X_i,X_j,\sigma_i,\sigma_j)$ for this model is
\begin{equation}
\label{eq:hmrflaw}
 \P(X_i\le x_i,X_j\le x_j,\sigma_i=s_i,\sigma_j=s_j) \\= p_{ij}(s_i,s_j)
 F^{s_i}_i(x_i)F^{s_j}_j(x_j).
\end{equation}
Since $\sigma_i$ and $\sigma_j$ are binary variables, $p_{ij}(s_i,s_j)$ can be expressed
with $3$ independent parameters,
\begin{align*}
p_{ij}(s_i,s_j) = p_{ij}^{11}s_is_j &+\left(p_{j}^1-p_{ij}^{11} \right)\bar s_i
s_j +\left(p_{i}^1-p_{ij}^{11} \right)s_i\bar s_j\\
&+\left(1-p_{i}^1-p_{j}^1+p_{ij}^{11}\right)\bar s_i\bar s_j,
\end{align*}
using the notation $\bar s \egaldef 1-s$ and with
\begin{align*}
p_i^1&\egaldef \P(\sigma_i=1) = {\mathbb E}{(\sigma_i)},\\
p_{ij}^{11} &\egaldef \P(\sigma_i=1,\sigma_j=1) =  {\mathbb E}{(\sigma_i
\sigma_j)}.
\end{align*}
The probability distribution is valid as soon as $(p_i^1,p_j^1) \in
[0,1]^2$ and 
\begin{equation*}
 p_{ij}^{11} \in \mathbb D(p_i^1,p_j^1) \egaldef \left[\max(0,p_i^1+p_j^1-1)
 ,\min(p_i^1,p_j^1)\right].
\end{equation*}
Until now, we have been able to make optimal choices in some sense, but obviously
the number of parameters is not enough to encode exactly any structure of dependency.
This is shown in the following proposition
\begin{proposition}\label{prop:capac}
 When the mutual information $I_{\hP}(X_i,X_j)$ between the real variables is
strictly greater
 than $\log(2)$, our model is not able to perfectly encode the joint distribution of
 $X_i$ and $X_j$ for any choice of encoding function.
\end{proposition}
This result is compatible with intuition: whatever the definition of the binary
variables, it will not be possible to share more than one bit of information between
two of them. However, we shall see in Section~\ref{sec:simul} that it is still
possible to obtain quasi-optimal performances for the prediction task even when the
mutual information is strictly greater than $\log(2)$.

\begin{proof}
 We will prove that the Kullback-Leibler divergence between the empirical joint
distribution $\hP$ of $(X_i,X_j)$ and the joint distribution $\P$ within our model
is strictly positive as soon as $I_{\hP}(X_i,X_j)>\log(2)$.
\begin{align*}
 D_\text{KL}(\hP||\P) &= \int \hP(x_i,x_j) \log \frac{\hP(x_i,x_j)}{\P(x_i,x_j)}
dx_idx_j\\
&= I_{\hP}(X_i,X_j) + \int \hP(x_i,x_j) \log
\frac{\hP(x_i)\hP(x_j)}{\DD\P(x_i,x_j)}dx_idx_j,\\
&\egaldef I_{\hP}(X_i,X_j) - \mathbb{I}(X_i,X_j).
\end{align*}
Using the fact that $\P(x_i) = \hP(x_i)$ and expanding w.r.t.\ $\sigma_i$ and
$\sigma_j$, one gets
\begin{align*}
 \mathbb{I}(X_i,X_j) &= \int \hP(x_i,x_j) \log \left(
\sum_{\sigma_i,\sigma_j} \frac{\P(\sigma_i,\sigma_j)}{\P(\sigma_i)\P(\sigma_j)}
\Lambda_i^{\sigma_i}(x_i)\Lambda_j^{\sigma_j}(x_j)\right)dx_idx_j,\\
&\leq \log \left(\int \hP(x_i,x_j) \sum_{\sigma_i,\sigma_j}
\frac{\P(\sigma_i,\sigma_j)}{\P(\sigma_i)\P(\sigma_j)}
\Lambda_i^{\sigma_i}(x_i)\Lambda_j^{\sigma_j}(x_j)dx_idx_j\right),
\end{align*}
with $\Lambda^1\egaldef\Lambda$ and $\Lambda^0 \egaldef 1-\Lambda$.
Defining $\hP_{\sigma_i\sigma_j} \egaldef \E_\hP[\Lambda_i^{\sigma_i}(X_i)
\Lambda_j^{\sigma_j}(X_j)]$, we get the final expression
\begin{equation*}
 \mathbb{I}(X_i,X_j) \leq \log \left( \sum_{\sigma_i,\sigma_j}
\frac{\P(\sigma_i,\sigma_j)}{\P(\sigma_i)\P(\sigma_j)}
\hP_{\sigma_i\sigma_j}\right) \leq \log(2),
\end{equation*}
because we have $\P(\sigma_i,\sigma_j) \leq \P(\sigma_j)$ and $\sum_{\sigma_j}
\hP_{\sigma_i\sigma_j} = \P(\sigma_i)$.
\end{proof}

\medskip

We will focus first on the estimation of the pairwise distribution $p_{ij}$ of
$(\sigma_i,\sigma_j)$, without discussing how to estimate the joint distribution of
$\bm{\sigma}$ from them. We will come back to this problem in the end of this
section.

\subsection{Pairwise distributions estimation}\label{ssec:PW_estim}

The choice of the encoding functions $\Lambda_i$ imposes the marginal distributions
of the latent variables $\sigma_i$; indeed we have seen that
\begin{equation*}
 p_i^1 = \P(\sigma_i=1) =\E[\Lambda_i(X_i)] = \int_{\X_i} \Lambda_i(x)dF_i(x).
\end{equation*}
These parameters can easily be estimated using empirical moments and it will only
remain to estimate the correlation parameter $p_{ij}^{11}$. We propose here to carry 
out a maximum likelihood estimation. The estimation of each parameter 
$p_{ij}^{11}$ is independent of the others and we carry out one unidimensional  
likelihood maximization per edge. For the sake of simplicity, we
assume that the random variables admit probability distribution functions. The
joint pdf of ($X_i,X_j$) associated to the distribution $p_{ij}$ will be referred to
as
\begin{equation*}
f^{ij}_{p_{ij}}(x_i,x_j) \egaldef \sum_{s_i,s_j}p_{ij}(s_i,s_j)f_i^{s_i}(x_i)
f_j^{s_j}(x_j),
\end{equation*}
where $f_i^{s_i}$ is the pdf associated to $dF_i^{s_i}$. Let us first
express the logarithm of the likelihood of a distribution $p_{ij}$ of
$(\sigma_i,\sigma_j)$ corresponding to the pairwise observations $\x$ described in
\eqref{eq:xobs}.
\begin{align*}
 L(\x,p_{ij}) &= \sum_{k=1}^{N_{ij}}
\log f_{p_{ij}}^{ij}(x_i^k,x_j^k)\\
&= \sum_{k=1}^{N_{ij}} \log \left(\sum_{s_i,s_j}
p_{ij}(s_i,s_j) f_i^{s_i}(x_i^k) f_j^{s_j}(x_j^k) \right).
\end{align*}
Because of the hidden variables $\sigma_i$ and $\sigma_j$, a sum appears within the
logarithms. Therefore, it will not be possible to find explicitly the distributions
 $p_{ij}$ maximizing $L(\x,p_{ij})$. The usual approach is to use the Expectation
Maximization algorithm (EM) first introduce by \citet{Dempster}. It consists in 
building a sequence of $(\sigma_i,\sigma_j)$-distribution $p_{ij}\n$ with increasing
likelihood. Using the following notation
\begin{align*}
 p_{ij}\n(s_i,s_j|x_i,x_j) &\egaldef
\P\n(\sigma_i=s_i,\sigma_j=s_j|X_i=x_i,X_j=x_j),
\end{align*}
the EM algorithm can be expressed as 
\begin{equation*}
 p_{ij}^{(n+1)} \gets \argmax_{p_{ij}}\quad \ell(p_{ij}||p_{ij}\n) \egaldef \sum_k
\sum_{s_i,s_j}p_{ij}\n(s_i,s_j|x_i^k,x_j^k)\log
p_{ij}(s_i,s_j),
\end{equation*}
The derivative of $\ell(p_{ij}||p_{ij}\n)$ with respect to $p_{ij}^{11}$ is
\begin{equation*}
 \frac{\partial\ell(p_{ij}||p_{ij}\n)}{\partial p_{ij}^{11}} =
 \sum_{k=1}^{N_{ij}}  \sum_{s_i,s_j}   \left(2\1_{\{s_i=s_j\}}-1\right)
  \frac{p_{ij}\n(s_i,s_j|x^k_i,x^k_j)}{p_{ij}(s_i,s_j)},\\
\end{equation*}
Stationary points yields an obvious solution, which is
\begin{equation*}
 p_{ij}(s_i,s_j) = \frac{1}{N_{ij}}\sum_{k=1}^{N_{ij}} p_{ij}\n(s_i,s_j|x_i^k,x_j^k).
\end{equation*}
The function that we maximize being concave, this solution is the unique stationary
point of $\ell(p_{ij}||p_{ij}\n)$. We obtain the following update rule for the EM
algorithm 
\begin{equation}\label{eq:EMglobal}
 p_{ij}^{(n+1)}(1,1) \gets \frac{1}{N_{ij}} \sum_{k=1}^{N_{ij}} 
 \frac{\psi_{ij}\n(1,1)
\Lambda_i(x_i^k)\Lambda_j(x_j^k)}{Z_{ij}(x_i^k,x_j^k)},
\end{equation}
with
\begin{equation*}
 \psi\n_{ij}(s_i,s_j) \egaldef \frac{p_{ij}\n(s_i,s_j)}{p_i\n(s_i)p_j\n(s_j)},
\end{equation*}
\begin{equation*}
 Z_{ij}(x_i,x_j) \egaldef \sum_{s_i,s_j}
\psi_{ij}\n(s_i,s_j)\Lambda_i^{s_i}(x_i)\Lambda_j^{s_j}(x_j),
\end{equation*}
and $\Lambda^1 \egaldef \Lambda$, $\Lambda^0 \egaldef 1-\Lambda$.
The update rule \eqref{eq:EMglobal} is quite simple, although one has to check that
the estimated parameter is valid, i.e.\ $p_{ij}^{11} \in \mathbb{D}(p_i^1,p_j^1)$.
If it is not the case, it means that the parameter saturates at one bound.

Now that we have proposed a way to estimate the pairwise marginal of the model, we
will focus in next section on how to estimate the Ising model of $\bm{\sigma}$ from
them.

\subsection{Latent Ising model estimation compatible with BP}\label{ssec:calibration}

We now return to the problem of estimating the joint distribution $\Psigma$ of
the random vector $\bm{\sigma}$ from its pairwise marginals $\{p_{ij}\}_{(i,j)\in\E}$.
First, let us remark that (as discussed by \citet{MacKay}) having
compatible marginals does not guaranty the existence of a joint distribution
$\Psigma$ such as 
\begin{equation*}
 \forall (i,j) \in \E, \forall s_i, \sum_{\>s_{\V\setminus \{i,j\}}} \Psigma(\bm{\sigma} 
 =\>s) =  p_{ij}(s_i,s_j),
\end{equation*}
However, in the case where the graph $\G=(\V,\E)$ contains no cycles, the joint
distribution is entirely determined by its pairwise marginals. This joint
distribution is expressed as
\begin{equation}\label{eq:InvBethe}
  \Psigma(\bm{\sigma} = \>s) = \prod_{(i,j)\in \E} 
  \frac{p_{ij}(s_i,s_j)}{p_i(s_i)p_j(s_j)} \prod_{i \in \V} p_i(s_i),
\end{equation}
with $p_i$ the marginal of $p_{ij}$ -- independent of $j$.
 
In the more general case of a graph containing cycles, the situation is more complex. This 
inverse Ising model is much studied in statistical physics (see \citet{CoMa} and
references within). Potentially it is NP-hard and can have no solution. Only
approximate methods can be used for graph of large size. \citet{Wain3} proposed an
approach of particular interest, which takes into account the fact that once the
distribution $\Psigma$ is fixed in an approximate way, the marginalisation will also
be performed in an approximate way. The idea is to use compatible approximations for these
two tasks. In our case, we wish to use the BP algorithm, described in
forthcoming Section~\ref{sec:bpalgo}, to compute the approximate marginals of $\Psigma$. It
seems reasonable to impose that, without any observation, the answer given by BP is
the historical marginals $\{p_{ij}\}$ and $\{p_i\}$. For doing so, the distribution
$\Psigma$ should be chosen under the Bethe approximation~\eqref{eq:InvBethe} which is
closely related to the BP algorithm, as we shall see in Section~\ref{sec:bpalgo}. If
this choice is a good candidate as starting point, the Bethe approximation is usually
too rough and overestimates correlations, and it is thus necessary to
improve on it. This can
be achieved using various results from linear response theory
(see~\citet{WeTe,Yasuda,MoMe}), when the level of correlation is not too high.

We use instead a simple but more robust approach, which is to modify the model
using a single parameter $\alpha$ such as 
\begin{equation}\label{eq:alphaBethe}
 \Psigma(\bm{\sigma} = \>s) = \prod_{(i,j)\in \E} 
  \left(\frac{p_{ij}(s_i,s_j)}{p_i(s_i)p_j(s_j)} \right)^\alpha \prod_{i \in \V}
p_i(s_i).
\end{equation}
$\alpha$ can roughly be interpreted as an inverse temperature, which role is to
avoid overcounting interactions when the graph contains cycles. This parameter
can easily be calibrated by finding a phase transition w.r.t.\ $\alpha$. Indeed, 
for $\alpha=0$, the BP output is exactly $\{p_i\}$ and, when $\alpha$ increases, it
remains close to it until some discontinuity appears (see \citet{FuLaAu}). In
some sense, the best $\alpha$ corresponds to the maximal interaction strength such
that the BP output remains close to $\{p_i\}$.

\section{A message passing inference algorithm}\label{sec:bpalgo}

According to the results of Section~\ref{sec:real2binary},
observations about the real-valued random variable $X_i$ are
converted into knowledge of the marginal distribution of $\sigma_i$.
In order to estimate the distributions of the others binary latent variables,
we need an inference algorithm allowing us to impose this
marginal constraint to node~$i$ when $X_i$ is observed. For this
task, we propose a modified version of the BP algorithm.

\subsection{The BP algorithm}

We present here the BP algorithm, first described by \citet{Pearl}, in
a way very similar to the one of \citet{YeFrWe3}.
We use in this section a slightly more general notation than in
Section~\ref{sec:intro}, since instead of considering only pairwise
interactions, variables in the set $\V$ interact through factors,
which are subsets $a\subset \V$ of variables. If $\F$ is
this set of factors, we consider the following probability measure
\begin{equation}
\label{eq:joint}
\P(\bm{\sigma}=\mathbf{s}) = \prod_{a \in \F} \psi_a(\mathbf{s}_a) \prod_{i \in
\V}\phi_i(s_i),
\end{equation}
where $\mathbf{s}_a=\{s_i, i\in a\}$. It is also possible to see
variables and factors as nodes of a same bipartite graph, in which case the
shorthand notation $i \in a$ should be interpreted as ``there is an
edge between $i$ and $a$''. $\F$ together with $\V$ define a factor
graph, such as defined by \citet{Kschi}. The set $\E$ of edges contains all the
couples $(a,i)\in\F\times\V$ such that $i\in a$. We denote by $d_i$ the degree of
the variable node $i$. The BP algorithm is a message passing procedure,
which output is a set of estimated marginal probabilities, the beliefs
$b_a(\textbf{s}_a)$ (including single nodes beliefs $b_i(s_i)$). The idea is to
factor the marginal probability at a given site as a product of contributions
coming from neighboring factor nodes, which are the messages. With
definition~\eqref{eq:joint} of the joint probability measure, the
updates rules read:
\begin{align}
m_{a\to i}(s_i) &\gets
\sum_{\mathbf{s}_{a\setminus i}} \psi_a(\mathbf{s}_a)\prod_{j\in a\setminus i}
n_{j\to a }(s_j), \label{urules}\\[0.2cm]
n_{i \to a}(s_i) &\egaldef \phi_i(s_i)\prod_{a'\ni i, a'\ne a}
m_{a'\to i}(s_i), \label{urulesn}
\end{align}
where the notation $\sum_{\mathbf{s}_a}$ should be understood as summing all
the variables $\sigma_i$, $i\in a\subset \V$, over the realizations
$s_i \in \{0,1\}$. In practice, the messages are often normalized so that
 $\sum_{s_i} m_{a\to i}(s_i)= 1$.

At any point of the algorithm, one can compute the current beliefs as
\begin{align}
b_i(s_i) &\egaldef
\frac{1}{Z_i}\phi_i(s_i)\prod_{a\ni i} m_{a\to i}(s_i),
\label{belief1}\\[0.2cm]
b_a(\mathbf{s}_a) &\egaldef
\frac{1}{Z_a}\psi_a(\mathbf{s}_a)\prod_{i\in a} n_{i\to a}(s_i),
\label{belief2}
\end{align}
where $Z_i$ and $Z_a$ are normalization constants that
ensure that
\begin{equation}
\label{eq:normb}
 \sum_{\sigma_i} b_i(\sigma_i) = 1,\qquad\sum_{\bm{\sigma}_a}b_a(\bm{\sigma}_a)
=1.
\end{equation}
When the algorithm has converged, the obtained beliefs $b_a$ and $b_i$
are compatible:
\begin{equation}
\sum_{\textrm{s}_{a \setminus i}} b_a(\textrm{s}_a) =
b_i(s_i).\label{eq:compat}
\end{equation}

\citet{YeFrWe3} proved that the belief propagation algorithm is an
iterative way of solving a variational problem: namely it minimizes
the Kullback-Leibler divergence $D_{KL}(b\|p)$ to the true probability measure
(\ref{eq:joint}) over all Bethe approximations on the factor graph, of
the form
\begin{equation*}
 b(\mathbf{s})
  = \prod_{a \in \F} \frac{b_a(\mathbf{s}_a)}{\prod_{i \in a}b_i(s_i)}
    \prod_{i \in \V} b_i(s_i),
\end{equation*}
subject to constraints~\eqref{eq:normb}--(\ref{eq:compat}).
The approximation is actually exact when the underlying graph is
a tree. The stationary points of the above variational problem are
beliefs at a fixed point of the BP algorithm (see~\citet{YeFrWe3}).
This alternative description of BP will be used in the next
section to derive a new variant of the algorithm.

\subsection{Imposing beliefs: mirror BP}\label{sec:incldata}

In the following, $\V^*$ will be the set of nodes $i$ such that $X_i$ is observed.
Assuming that the model ($\psi_a$ and $\phi_i$) is given, we wish to include
in the algorithm some constraints on the beliefs of the form
\begin{equation}
\label{eq:know_cons}
\forall i \in \mathbb{V}^*, \forall s_i \in \{0,1\}, b_i(s_i) =
b^*_i(s_i).
\end{equation}
We suppose in the following that each $b^*_i$ is normalized. The issue
of how to convert real-valued observation to this distribution
$b_i^*$ has been studied in Section~\ref{sec:real2binary}. We seek to
obtain a new update rule from the Kullback-Leibler divergence
minimization, with the additional constraints~\eqref{eq:know_cons}. Constraints of
this form as sometimes referred to as ``soft constraints'' in the
Bayesian community (\citet{Bilmes}).

We start from the Lagrangian of the minimization problem:
\begin{align*}
\mathcal{L}(b,\lambda) = &D_{KL}(b\|p) + \sum_{\substack{i \in \V\setminus \V^*\\ a\ni i, s_i}}
\lambda_{ai}(s_i)\Bigl(b_i(s_i)  - \sum_{\mathbf{s}_{a\setminus i}}b_a(\mathbf{
s}_a) \Bigr)\\
&+ \sum_{\substack{i \in \V^*\\a\ni i, s_i}} \lambda_{ai}(s_i)
\Bigl(b^*_i(s_i) - \sum_{\mathbf{s}_{a\setminus i}}b_a(\mathbf{s}_a) \Bigr) + \sum_{i
\in\V} \gamma_i \left(\sum_{s_i}b_i(s_i)-1\right),
\end{align*}
with $D_{KL}(b\|p)$ defined as
\begin{equation*}
D_{KL}(b\|p)
  \egaldef \sum_{a,\bm{s}_a} b_a(\bm{s}_a) \log\frac{b_a(\bm{s}_a)}{\psi_a(\bm{
s}_a)} + \sum_{i,s_i} b_i(s_i) \log \frac{b_i(s_i)^{1-d_i}}{\phi_i(s_i)}.
\end{equation*}
The stationary points satisfy
\begin{align*}
\left\{\begin{array}{l}
b_a(\mathbf{s}_a)= \psi_a(\mathbf{s}_a) \exp\Bigl(\sum_{i\in a}\lambda_{ai}(s_
i)
- 1\Bigr),\,\forall a \in \F, \\[0.2cm]
b_i(s_i)= \phi_i(s_i)
\exp\Bigl(\frac{\sum_{a\ni i}\lambda_{ai}(s_i)}{d_i-1}\!+\!1\!-\!\gamma_i\Bigr)
,\forall i \notin \V^{\!*},\\[0.2cm]
b_i(s_i)= b^*_i(s_i), \, \forall i \in \V^*.
\end{array}\right.
\end{align*}
Following~\citet{YeFrWe3}, we introduce the parametrization
\begin{equation*}
\lambda_{ai}(s_i) = \log n_{i\to a}(s_i),
\end{equation*}
for all edges $(ai) \in \E$. Note that we do not consider any node in
$\V\setminus\V^*$ of degree $d_i$ equal to $1$ since they play no role in
the minimization problem. For all nodes $i\notin\V^*$ we
also have
\begin{equation*}
\lambda_{ai}(s_i) = \log \Bigl[ \phi_i(s_i)\prod_{b\ni i, b\ne a} m_{b\to
i}(s_i)\Bigr].
\end{equation*}
Then it follows that, whenever $i \notin \V^*$,
\begin{equation*}
n_{i\to a} (s_i) = \phi_i(s_i) \prod_{b\ni i,b\ne a} m_{b\to i}(s_i).
\end{equation*}
Enforcing the compatibility constraints on nodes $i \not\in \V^*$ shows that update
rules \eqref{urules}--(\ref{urulesn}) are still valid for these nodes. For $i \in
\V^*$, the compatibility constraints yield
\begin{align*}
b^*_i(s_i)&=\sum_{\mathbf{s}_{a\setminus i}} b_a(\mathbf{s}_a) =
\sum_{\mathbf{s}_{a\setminus i}} \psi_a(\mathbf{s}_a)
\prod_{j\in a}n_{j\to a}(s_j)\\
&= n_{i \to a}(s_i) \sum_{\mathbf{s}_{a\setminus i}} \psi_a(\mathbf{s}_a)\prod_{
j\in a}n_{j\to a}(s_j).
\end{align*}
Until now the message from a factor $a$ to a variable $i \in \V^*$ has not been
defined. For convenience we define it as in \eqref{urules} and the preceding
equation becomes
\begin{equation*}
 n_{i \to a}(s_i) m_{a \to i}(s_i) = b^*_i(s_i),
\end{equation*}
as in the usual BP algorithm. This leads to a definition that replaces
\eqref{urulesn} when $i\in\V^*$
\begin{equation}
\label{eq:nuppp}
 n_{i \to a}(s_i) = \frac{b^*_i(s_i)}{m_{a \to
i}(s_i)} = \frac{b_i^*(s_i)}{b_i(x_i)}\phi_i(s_i) \prod_{b \ni i,b \ne a} m_{b
\to i}(s_i).
\end{equation}
Therefore the message~\eqref{eq:nuppp} is the BP message~\eqref{urulesn}
multiplied by the ratio of the belief we are imposing over the ``current belief''
computed using~\eqref{belief1}. This is very similar to iterative proportional
fitting (IPF, see \citet{Darroch}). To sum up, the characteristics of this new
variant of Belief Propagation are
\begin{itemize}
\item all factors and all variables which value has not been fixed
  send the same messages \eqref{urules}--(\ref{urulesn}) as in classic BP;
\item variables which value has been fixed use the new
  messages~\eqref{eq:nuppp};
\item beliefs for factors or for variables which value has not been
fixed are still computed using \eqref{belief1}--(\ref{belief2});
\end{itemize}

In the classical BP algorithm, the information sent by one node can
only go back to itself through a cycle of the graph. When
\eqref{eq:nuppp} is used, however, the variable with fixed value acts
like a mirror and sends back the message to the factor instead of
propagating it through the graph. It is to emphasize this property
that we call our new method the \emph{mirror BP} (mBP) algorithm. Note that it could
be defined for variables valued in any discrete alphabet.

A very similar algorithm to our mBP has been proposed by~\citet{WeTe2}.
Their algorithm is described as iterations of successive BP runs on unobserved
nodes and IPF on nodes in $\V^*$. The update \eqref{eq:nuppp} is just obtained
as direct IPF\@. The main drawback of their version is that it assumes a
particular update ordering because they consider that the updates 
\eqref{urulesn}--(\ref{eq:nuppp}) are of different nature, which is in fact not
really necessary and complicates its use.

\bigskip

It is known that BP can exhibit non convergent behavior in loopy networks, although
sufficient conditions for convergence are known (see
e.g.\ \citet{MooijKappen07,Tatikonda02,Ihler}). Since the mirroring behavior
of our algorithm seems to be quite different, we present some sufficient conditions
for convergence.
\begin{Def}
 Let $\mathcal{T}(\mathcal{G},\V^*)$ be the factor graph where each node $i \in \V^*$
has been cloned $d_i$ times, each clone being attached to one (and only one) neighbor
of $i$. We call ``graph cutting at $\V^*$'' the transformation
$\mathcal{T}(\cdot,\V^*)$ applied to a factor graph $\mathcal{G}$ for a given
set of variable nodes $\V^*$.
\end{Def}

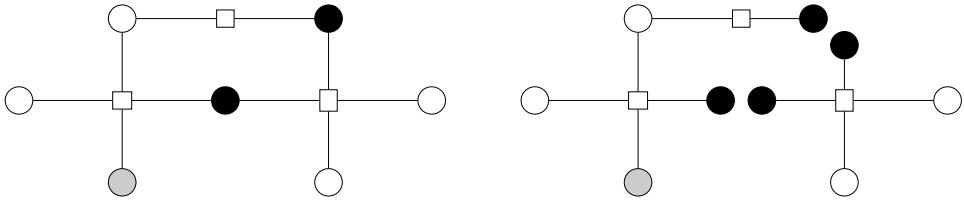
\begin{figure}
\resizebox{\textwidth}{!}{\begin{tikzpicture}[scale=1]

    \path (5,-1) node[draw,shape=circle,fill=black] (s2) {$\phantom{2}$};
    \path (0,-1) node[draw,shape=circle] (s1) {$\phantom 1$};
    \path (10,-1) node[draw,shape=circle] (s3){$\phantom 3$};
    \path (2.5,-3) node[draw,shape=circle,fill=black!20] (s4){$\phantom 4$};
    \path (7.5,-3) node[draw,shape=circle] (s5){$\phantom 5$};
    \path (2.5,1) node[draw,shape=circle] (s6){$\phantom 6$};
    \path (7.5,1) node[draw,shape=circle,fill=black] (s7){$\phantom 7$};
    
    \path (2.5,-1) node[draw,shape=rectangle] (fa) {$\phantom a$};
    \path (7.5,-1) node[draw,shape=rectangle] (fb) {$\phantom b$};
    \path (5,1) node[draw,shape=rectangle] (fd) {$\phantom c$};

    \path[-] (s1) edge (fa);
    \path[-] (fa) edge (s2);
    \path[-] (fa) edge (s4);
    \path[-] (fb) edge (s2);
    \path[-] (fb) edge (s5);
    \path[-] (fb) edge (s3);
    \path[-] (fa) edge (s6);
    \path[-] (fd) edge (s7);
    \path[-] (fd) edge (s6);
    \path[-] (fb) edge (s7);

    \path (17,-1) node[draw,shape=circle,fill=black] (s2p) {$\phantom 2$};
    \path (18,-1) node[draw,shape=circle,fill=black] (s2pp) {$\phantom 2$};
    \path (12.5,-1) node[draw,shape=circle] (s1p) {$\phantom 1$};
    \path (22.5,-1) node[draw,shape=circle] (s3p){$\phantom 3$};
    \path (20,-3) node[draw,shape=circle] (s5p){$\phantom5$};
    \path (15,-1) node[draw,shape=rectangle] (fap) {$\phantom a$};
    \path (20,-1) node[draw,shape=rectangle] (fbp) {$\phantom b$};
    \path (15,1) node[draw,shape=circle] (s6p){$\phantom6$};
    \path (15,-3) node[draw,shape=circle,fill=black!20] (s4p){$\phantom4$};
    \path (19.25,1) node[draw,shape=circle,fill=black] (s7p){$\phantom7$};
    \path (20,0.35) node[draw,shape=circle,fill=black] (s7pp){$\phantom7$};
    \path (17.5,1) node[draw,shape=rectangle] (fdp) {$\phantom c$};

    \path[-] (fap) edge (s1p);
    \path[-] (fap) edge (s6p);
    \path[-] (fdp) edge (s6p);
    \path[-] (fdp) edge (s7p);
    \path[-] (fbp) edge (s7pp);
    \path[-] (fbp) edge (s3p);
    \path[-] (fbp) edge (s5p);
    \path[-] (fap) edge (s2p);
    \path[-] (fbp) edge (s2pp);
    \path[-] (fap) edge (s4p);

\end{tikzpicture}}
 \caption{Illustration of Proposition~\ref{prop:stability}. If only black nodes
 are in $\V^*$, Proposition~\ref{prop:stability} tells us that mBP converges
 since the resulting graph $\mathcal{T}(\mathcal{G},\V^*)$
(right graph) contains two disconnected trees with exactly two nodes in
$\V^*$. If we add the gray node in $\V^*$ then Proposition~\ref{prop:stability}
does not apply, the right tree contains three nodes in $\V^*$, and we
cannot conclude about convergence. However, on the other part of the graph
Proposition~\ref{prop:stability} still holds.}
\label{fig:exconv}
\end{figure}
Example of a such ``graph cutting'' $\mathcal{T}(\mathcal{G},\V^*)$ is shown in
Figure~\ref{fig:exconv}. The following proposition describes cases where the mBP
algorithm is guaranteed to converge.

\begin{proposition}
\label{prop:stability}
If the graph $\mathcal{T}(\mathcal{G},\V^*)$ is formed by disconnected
trees containing not more than two leaves cloned from $\V^*$, the
mBP algorithm is stable and converges to a unique fixed point.
\end{proposition}
\begin{proof}
See Appendix.
\end{proof}

\section{Numerical experiments}\label{sec:simul}

In order to understand its behavior, we apply here the method described
in this paper to synthetic data. We will consider three cases of increasing complexity: 
\begin{itemize}
 \item a pair ($X_1,X_2$) of real-valued random variables,
 \item a tree with interior connectivity fixed,
 \item a rough road traffic network.
\end{itemize}

For each case, we repeat the following \emph{decimation experiment}: for an outcome
of the random vector $\>X$, we observe its components $X_i$ in a random order and we
make prediction about unobserved components using our method. This will allow
us to compare the performance of the different choices of encoding and
decoding functions when the proportion of observed variables varies.
 
We will consider the following choices for the encoding and decoding functions:
\begin{itemize}
 \item the step function $\Lambda_\text{MI}$ with the Bayesian decoding function
       \eqref{eq:GpMI},
 \item the cumulative distribution function $F$ with its inverse $\GammaD =
       F^{-1}$,
 \item the cumulative distribution function $F$ with the decoding
       function $\GammaP$ of~\eqref{eq:GpS}.
\end{itemize}
Each of these choices yields an estimator $\theta$ for which we will compute the
performance w.r.t.\ the $L^1$ norm
\begin{equation}\label{eq:L1perf}
 \E_X\Bigl[|\theta(X) - X|\Bigr].
\end{equation}

\paragraph{Model generation.}
These synthetic models are based on Gaussian copulas with support corresponding to one of 
the three cases previously described. More precisely, it corresponds to the support of
the precision matrix, i.e.\ the inverse covariance matrix, of the Gaussian vector $\>Y$.
For doing so, we randomly generate the partial correlations, the entries of the precision
matrix of $\>Y$, with uniform random variables on $[-1,-0.2]\cup[0.2,1]$. Since this will
not always lead to a positive definite precision matrix, we use this matrix as a
starting point and reduce the highest correlation until it becomes definite positive.

We can then generate outcomes of this Gaussian vector $\>Y$ and transform them, using
the function which maps a Gaussian variable $\mathcal{N}(0,1)$ into a variable of
chosen cdf $F_X$. More precisely, each component of the vector $\>X$ is defined as
\begin{equation*}
 X_i = F^{-1}_X\left(F_{\N(0,1)}(Y_i)\right),
\end{equation*}
where $F_{\N(0,1)}$ is the cdf of a $\N(0,1)$ variable. This procedure will allow
us to perform exact inference using the Gaussian vector $\>Y$ while the dependency 
of the vector $\>X$ is based on a Gaussian copula. 

We will sometimes consider the case of $\beta(a,b)$ variables,
so let us recall their pdfs $f^\beta_{a,b}$ for $a,b \in ]0,+\infty[$ 
\begin{equation*}
  f^\beta_{a,b}(x) = \frac{1}{B(a,b)}x^{a-1}(1-x)^{b-1}\1_{[0,1]}(x),
\end{equation*}
where the normalization constant $B(a,b)$ is the bêta function.
These distributions are of particular interest because different cases arise
depending on the parameters $a$ and $b$. Indeed, it is possible to obtain almost
binary ($a,b\to 0$), unimodal ($a,b>1$) or uniform ($a,b=1$) distributions on
$[0,1]$.

\paragraph{A pair $(X_1,X_2)$ of real-valued random variables.}
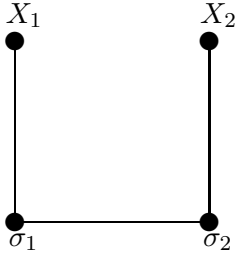
\begin{figure}
\centering
\setlength{\unitlength}{0.8mm}
\begin{picture}(30, 42)(0,0)
  \put(0,7){\color{black}{\circle*{3}}} 
  \put(-1,3){$\sigma_1$}
  \put(1.5,7){\line(1,0){30}}
  \put(32,7){\color{black}{\circle*{3}}} 
  \put(31,3){$\sigma_2$}
  \put(0,37){\color{black}{\circle*{3}}} 
  \put(-1.5,40){$X_1$}
  \put(0,8.5){\line(0,1){27}}
  \put(32,37){\color{black}{\circle*{3}}} 
  \put(30.5,40){$X_2$}
  \put(32,8.5){\line(0,1){27}}
\end{picture}
\caption{Model of the random vector $(X_1,X_2,\sigma_1,\sigma_2)$. The true
distribution of $(X_1,X_2)$ is approximated through the latent
variables $\sigma_1$ and $\sigma_2$.}
\label{fig:simple_PW}
\end{figure}
\begin{table}
\centering
\begin{tabular}{@{}cccccc@{}} \toprule
Density & Correlation  & $(F,\Gamma^\mathcal{D})$ & $(F,\Gamma^\mathcal{P})$ &
$\Lambda_\mathrm{MI}$ & Exact \\ \bottomrule
\multirow{4}{*}[-2.5pt]{\resizebox{0.16\textwidth}{!}{
\begingroup
  \makeatletter
  \providecommand\color[2][]{%
    \GenericError{(gnuplot) \space\space\space\@spaces}{%
      Package color not loaded in conjunction with
      terminal option `colourtext'%
    }{See the gnuplot documentation for explanation.%
    }{Either use 'blacktext' in gnuplot or load the package
      color.sty in LaTeX.}%
    \renewcommand\color[2][]{}%
  }%
  \providecommand\includegraphics[2][]{%
    \GenericError{(gnuplot) \space\space\space\@spaces}{%
      Package graphicx or graphics not loaded%
    }{See the gnuplot documentation for explanation.%
    }{The gnuplot epslatex terminal needs graphicx.sty or graphics.sty.}%
    \renewcommand\includegraphics[2][]{}%
  }%
  \providecommand\rotatebox[2]{#2}%
  \@ifundefined{ifGPcolor}{%
    \newif\ifGPcolor
    \GPcolortrue
  }{}%
  \@ifundefined{ifGPblacktext}{%
    \newif\ifGPblacktext
    \GPblacktexttrue
  }{}%
  \let\gplgaddtomacro\g@addto@macro
  \gdef\gplbacktext{}%
  \gdef\gplfronttext{}%
  \makeatother
  \ifGPblacktext
    \def\colorrgb#1{}%
    \def\colorgray#1{}%
  \else
    \ifGPcolor
      \def\colorrgb#1{\color[rgb]{#1}}%
      \def\colorgray#1{\color[gray]{#1}}%
      \expandafter\def\csname LTw\endcsname{\color{white}}%
      \expandafter\def\csname LTb\endcsname{\color{black}}%
      \expandafter\def\csname LTa\endcsname{\color{black}}%
      \expandafter\def\csname LT0\endcsname{\color[rgb]{1,0,0}}%
      \expandafter\def\csname LT1\endcsname{\color[rgb]{0,1,0}}%
      \expandafter\def\csname LT2\endcsname{\color[rgb]{0,0,1}}%
      \expandafter\def\csname LT3\endcsname{\color[rgb]{1,0,1}}%
      \expandafter\def\csname LT4\endcsname{\color[rgb]{0,1,1}}%
      \expandafter\def\csname LT5\endcsname{\color[rgb]{1,1,0}}%
      \expandafter\def\csname LT6\endcsname{\color[rgb]{0,0,0}}%
      \expandafter\def\csname LT7\endcsname{\color[rgb]{1,0.3,0}}%
      \expandafter\def\csname LT8\endcsname{\color[rgb]{0.5,0.5,0.5}}%
    \else
      \def\colorrgb#1{\color{black}}%
      \def\colorgray#1{\color[gray]{#1}}%
      \expandafter\def\csname LTw\endcsname{\color{white}}%
      \expandafter\def\csname LTb\endcsname{\color{black}}%
      \expandafter\def\csname LTa\endcsname{\color{black}}%
      \expandafter\def\csname LT0\endcsname{\color{black}}%
      \expandafter\def\csname LT1\endcsname{\color{black}}%
      \expandafter\def\csname LT2\endcsname{\color{black}}%
      \expandafter\def\csname LT3\endcsname{\color{black}}%
      \expandafter\def\csname LT4\endcsname{\color{black}}%
      \expandafter\def\csname LT5\endcsname{\color{black}}%
      \expandafter\def\csname LT6\endcsname{\color{black}}%
      \expandafter\def\csname LT7\endcsname{\color{black}}%
      \expandafter\def\csname LT8\endcsname{\color{black}}%
    \fi
  \fi
  \setlength{\unitlength}{0.0500bp}%
  \begin{picture}(7200.00,5040.00)%
    \gplgaddtomacro\gplbacktext{%
      \csname LTb\endcsname%
      \put(1440,3527){\makebox(0,0)[l]{\strut{}\fontsize{45}{45}\selectfont$\beta(\nicefrac{1}{10},\nicefrac{1}{10})$}}%
    }%
    \gplgaddtomacro\gplfronttext{%
    }%
    \gplbacktext
    \put(0,0){\includegraphics{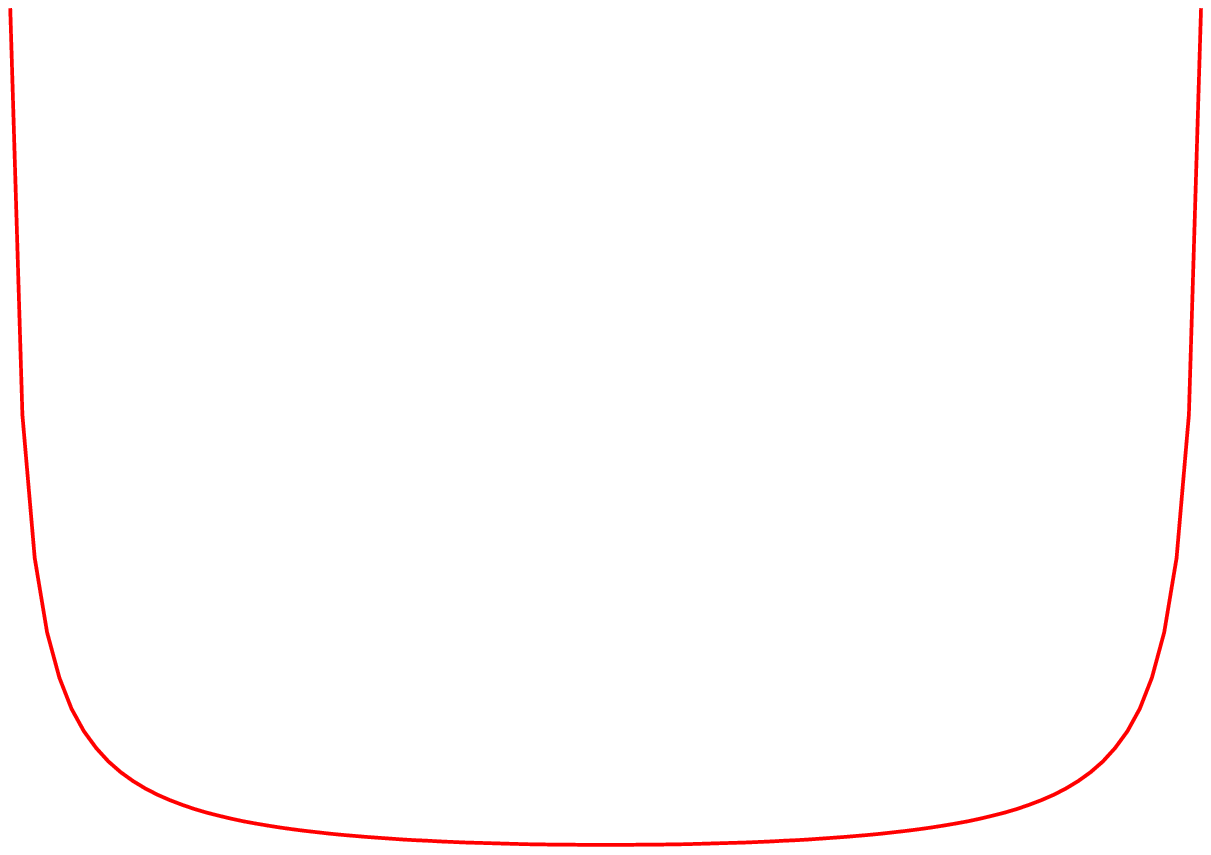}}%
    \gplfronttext
  \end{picture}%
\endgroup
}}
& \multirow{2}{*}{$\rho=0.5$} &
\textbf{\numprint{0.2}\%} &
\numprint{3.1}\% & \numprint{3.3}\% &  \numprint{32.04}\\ 
& & \numprint{-0.34} & \numprint{-0.79} & \numprint{0.01}& \numprint{0}\\
\cmidrule{2-6}
&\multirow{2}{*}{$\rho=0.9$} & \textbf{\numprint{0.1}\%} & \numprint{22}\% &
\numprint{9.5}\% & \numprint{13.62}\\
& & \numprint{-0.15} & \numprint{-0.72} & \numprint{-0.06} & \numprint{0}
\\\bottomrule

\multirow{4}{*}[-2.5pt]{\resizebox{0.16\textwidth}{!}{
\begingroup
  \makeatletter
  \providecommand\color[2][]{%
    \GenericError{(gnuplot) \space\space\space\@spaces}{%
      Package color not loaded in conjunction with
      terminal option `colourtext'%
    }{See the gnuplot documentation for explanation.%
    }{Either use 'blacktext' in gnuplot or load the package
      color.sty in LaTeX.}%
    \renewcommand\color[2][]{}%
  }%
  \providecommand\includegraphics[2][]{%
    \GenericError{(gnuplot) \space\space\space\@spaces}{%
      Package graphicx or graphics not loaded%
    }{See the gnuplot documentation for explanation.%
    }{The gnuplot epslatex terminal needs graphicx.sty or graphics.sty.}%
    \renewcommand\includegraphics[2][]{}%
  }%
  \providecommand\rotatebox[2]{#2}%
  \@ifundefined{ifGPcolor}{%
    \newif\ifGPcolor
    \GPcolortrue
  }{}%
  \@ifundefined{ifGPblacktext}{%
    \newif\ifGPblacktext
    \GPblacktexttrue
  }{}%
  \let\gplgaddtomacro\g@addto@macro
  \gdef\gplbacktext{}%
  \gdef\gplfronttext{}%
  \makeatother
  \ifGPblacktext
    \def\colorrgb#1{}%
    \def\colorgray#1{}%
  \else
    \ifGPcolor
      \def\colorrgb#1{\color[rgb]{#1}}%
      \def\colorgray#1{\color[gray]{#1}}%
      \expandafter\def\csname LTw\endcsname{\color{white}}%
      \expandafter\def\csname LTb\endcsname{\color{black}}%
      \expandafter\def\csname LTa\endcsname{\color{black}}%
      \expandafter\def\csname LT0\endcsname{\color[rgb]{1,0,0}}%
      \expandafter\def\csname LT1\endcsname{\color[rgb]{0,1,0}}%
      \expandafter\def\csname LT2\endcsname{\color[rgb]{0,0,1}}%
      \expandafter\def\csname LT3\endcsname{\color[rgb]{1,0,1}}%
      \expandafter\def\csname LT4\endcsname{\color[rgb]{0,1,1}}%
      \expandafter\def\csname LT5\endcsname{\color[rgb]{1,1,0}}%
      \expandafter\def\csname LT6\endcsname{\color[rgb]{0,0,0}}%
      \expandafter\def\csname LT7\endcsname{\color[rgb]{1,0.3,0}}%
      \expandafter\def\csname LT8\endcsname{\color[rgb]{0.5,0.5,0.5}}%
    \else
      \def\colorrgb#1{\color{black}}%
      \def\colorgray#1{\color[gray]{#1}}%
      \expandafter\def\csname LTw\endcsname{\color{white}}%
      \expandafter\def\csname LTb\endcsname{\color{black}}%
      \expandafter\def\csname LTa\endcsname{\color{black}}%
      \expandafter\def\csname LT0\endcsname{\color{black}}%
      \expandafter\def\csname LT1\endcsname{\color{black}}%
      \expandafter\def\csname LT2\endcsname{\color{black}}%
      \expandafter\def\csname LT3\endcsname{\color{black}}%
      \expandafter\def\csname LT4\endcsname{\color{black}}%
      \expandafter\def\csname LT5\endcsname{\color{black}}%
      \expandafter\def\csname LT6\endcsname{\color{black}}%
      \expandafter\def\csname LT7\endcsname{\color{black}}%
      \expandafter\def\csname LT8\endcsname{\color{black}}%
    \fi
  \fi
  \setlength{\unitlength}{0.0500bp}%
  \begin{picture}(7200.00,5040.00)%
    \gplgaddtomacro\gplbacktext{%
      \csname LTb\endcsname%
      \put(1440,3527){\makebox(0,0)[l]{\strut{}\fontsize{45}{45}\selectfont$\beta(2,3)$}}%
    }%
    \gplgaddtomacro\gplfronttext{%
    }%
    \gplbacktext
    \put(0,0){\includegraphics{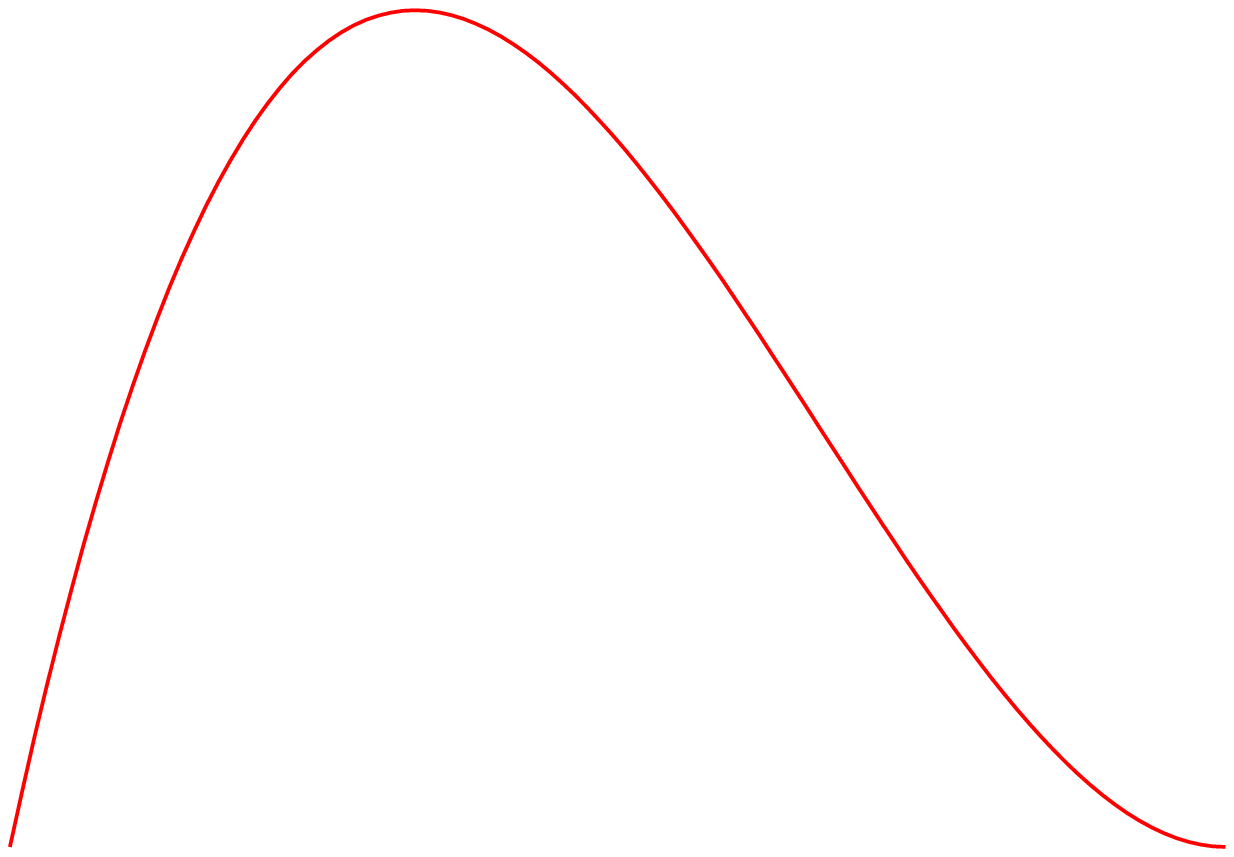}}%
    \gplfronttext
  \end{picture}%
\endgroup
}}
&\multirow{2}{*}{$\rho=0.5$} & \textbf{\numprint{1.4}\%} & \numprint{4.4}\% &
\numprint{6.7}\% &  \numprint{14.22}\\
& & \numprint{-0.26}& \numprint{0.47} & \numprint{0.44}& 
\numprint{0}\\ \cmidrule{2-6}
&\multirow{2}{*}{$\rho=0.9$} & \textbf{\numprint{1.3}\%} & \numprint{68.8}\% &
\numprint{61.7}\% &  \numprint{6.89}\\
& & \numprint{0.1} & \numprint{1.25} & \numprint{0.92} &
\numprint{0}\\\bottomrule

\multirow{4}{*}[-2.5pt]{\resizebox{0.16\textwidth}{!}{
\begingroup
  \makeatletter
  \providecommand\color[2][]{%
    \GenericError{(gnuplot) \space\space\space\@spaces}{%
      Package color not loaded in conjunction with
      terminal option `colourtext'%
    }{See the gnuplot documentation for explanation.%
    }{Either use 'blacktext' in gnuplot or load the package
      color.sty in LaTeX.}%
    \renewcommand\color[2][]{}%
  }%
  \providecommand\includegraphics[2][]{%
    \GenericError{(gnuplot) \space\space\space\@spaces}{%
      Package graphicx or graphics not loaded%
    }{See the gnuplot documentation for explanation.%
    }{The gnuplot epslatex terminal needs graphicx.sty or graphics.sty.}%
    \renewcommand\includegraphics[2][]{}%
  }%
  \providecommand\rotatebox[2]{#2}%
  \@ifundefined{ifGPcolor}{%
    \newif\ifGPcolor
    \GPcolortrue
  }{}%
  \@ifundefined{ifGPblacktext}{%
    \newif\ifGPblacktext
    \GPblacktexttrue
  }{}%
  \let\gplgaddtomacro\g@addto@macro
  \gdef\gplbacktext{}%
  \gdef\gplfronttext{}%
  \makeatother
  \ifGPblacktext
    \def\colorrgb#1{}%
    \def\colorgray#1{}%
  \else
    \ifGPcolor
      \def\colorrgb#1{\color[rgb]{#1}}%
      \def\colorgray#1{\color[gray]{#1}}%
      \expandafter\def\csname LTw\endcsname{\color{white}}%
      \expandafter\def\csname LTb\endcsname{\color{black}}%
      \expandafter\def\csname LTa\endcsname{\color{black}}%
      \expandafter\def\csname LT0\endcsname{\color[rgb]{1,0,0}}%
      \expandafter\def\csname LT1\endcsname{\color[rgb]{0,1,0}}%
      \expandafter\def\csname LT2\endcsname{\color[rgb]{0,0,1}}%
      \expandafter\def\csname LT3\endcsname{\color[rgb]{1,0,1}}%
      \expandafter\def\csname LT4\endcsname{\color[rgb]{0,1,1}}%
      \expandafter\def\csname LT5\endcsname{\color[rgb]{1,1,0}}%
      \expandafter\def\csname LT6\endcsname{\color[rgb]{0,0,0}}%
      \expandafter\def\csname LT7\endcsname{\color[rgb]{1,0.3,0}}%
      \expandafter\def\csname LT8\endcsname{\color[rgb]{0.5,0.5,0.5}}%
    \else
      \def\colorrgb#1{\color{black}}%
      \def\colorgray#1{\color[gray]{#1}}%
      \expandafter\def\csname LTw\endcsname{\color{white}}%
      \expandafter\def\csname LTb\endcsname{\color{black}}%
      \expandafter\def\csname LTa\endcsname{\color{black}}%
      \expandafter\def\csname LT0\endcsname{\color{black}}%
      \expandafter\def\csname LT1\endcsname{\color{black}}%
      \expandafter\def\csname LT2\endcsname{\color{black}}%
      \expandafter\def\csname LT3\endcsname{\color{black}}%
      \expandafter\def\csname LT4\endcsname{\color{black}}%
      \expandafter\def\csname LT5\endcsname{\color{black}}%
      \expandafter\def\csname LT6\endcsname{\color{black}}%
      \expandafter\def\csname LT7\endcsname{\color{black}}%
      \expandafter\def\csname LT8\endcsname{\color{black}}%
    \fi
  \fi
  \setlength{\unitlength}{0.0500bp}%
  \begin{picture}(7200.00,5040.00)%
    \gplgaddtomacro\gplbacktext{%
      \csname LTb\endcsname%
     
\put(1440,3527){\makebox(0,0)[l]{\strut{}\fontsize{45}{45}\selectfont$\beta(1,1)$}}%

    }%
    \gplgaddtomacro\gplfronttext{%
    }%
    \gplbacktext
    \put(0,0){\includegraphics{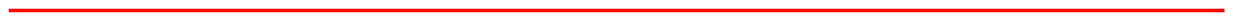}}%
    \gplfronttext
  \end{picture}%
\endgroup
}}
&\multirow{2}{*}{$\rho=0.5$} & \textbf{\numprint{0.1}\%} & \numprint{4.6}\% &
\numprint{7.3}\% &  \numprint{20.96}\\ 
& & \numprint{0.02}& \numprint{0.01} & \numprint{-0.03}& 
\numprint{0}\\ \cmidrule{2-6}
&\multirow{2}{*}{$\rho=0.9$} & \textbf{\numprint{0.4}\%} & \numprint{66.2}\% &
\numprint{58.6}\% &  \numprint{9.83}\\
& &  \numprint{0.02} & \numprint{-0.09} & \numprint{0} &  \numprint{0} \\\bottomrule

\multirow{4}{*}[-2.5pt]{\resizebox{0.16\linewidth}{!}{
\begingroup
  \makeatletter
  \providecommand\color[2][]{%
    \GenericError{(gnuplot) \space\space\space\@spaces}{%
      Package color not loaded in conjunction with
      terminal option `colourtext'%
    }{See the gnuplot documentation for explanation.%
    }{Either use 'blacktext' in gnuplot or load the package
      color.sty in LaTeX.}%
    \renewcommand\color[2][]{}%
  }%
  \providecommand\includegraphics[2][]{%
    \GenericError{(gnuplot) \space\space\space\@spaces}{%
      Package graphicx or graphics not loaded%
    }{See the gnuplot documentation for explanation.%
    }{The gnuplot epslatex terminal needs graphicx.sty or graphics.sty.}%
    \renewcommand\includegraphics[2][]{}%
  }%
  \providecommand\rotatebox[2]{#2}%
  \@ifundefined{ifGPcolor}{%
    \newif\ifGPcolor
    \GPcolortrue
  }{}%
  \@ifundefined{ifGPblacktext}{%
    \newif\ifGPblacktext
    \GPblacktexttrue
  }{}%
  \let\gplgaddtomacro\g@addto@macro
  \gdef\gplbacktext{}%
  \gdef\gplfronttext{}%
  \makeatother
  \ifGPblacktext
    \def\colorrgb#1{}%
    \def\colorgray#1{}%
  \else
    \ifGPcolor
      \def\colorrgb#1{\color[rgb]{#1}}%
      \def\colorgray#1{\color[gray]{#1}}%
      \expandafter\def\csname LTw\endcsname{\color{white}}%
      \expandafter\def\csname LTb\endcsname{\color{black}}%
      \expandafter\def\csname LTa\endcsname{\color{black}}%
      \expandafter\def\csname LT0\endcsname{\color[rgb]{1,0,0}}%
      \expandafter\def\csname LT1\endcsname{\color[rgb]{0,1,0}}%
      \expandafter\def\csname LT2\endcsname{\color[rgb]{0,0,1}}%
      \expandafter\def\csname LT3\endcsname{\color[rgb]{1,0,1}}%
      \expandafter\def\csname LT4\endcsname{\color[rgb]{0,1,1}}%
      \expandafter\def\csname LT5\endcsname{\color[rgb]{1,1,0}}%
      \expandafter\def\csname LT6\endcsname{\color[rgb]{0,0,0}}%
      \expandafter\def\csname LT7\endcsname{\color[rgb]{1,0.3,0}}%
      \expandafter\def\csname LT8\endcsname{\color[rgb]{0.5,0.5,0.5}}%
    \else
      \def\colorrgb#1{\color{black}}%
      \def\colorgray#1{\color[gray]{#1}}%
      \expandafter\def\csname LTw\endcsname{\color{white}}%
      \expandafter\def\csname LTb\endcsname{\color{black}}%
      \expandafter\def\csname LTa\endcsname{\color{black}}%
      \expandafter\def\csname LT0\endcsname{\color{black}}%
      \expandafter\def\csname LT1\endcsname{\color{black}}%
      \expandafter\def\csname LT2\endcsname{\color{black}}%
      \expandafter\def\csname LT3\endcsname{\color{black}}%
      \expandafter\def\csname LT4\endcsname{\color{black}}%
      \expandafter\def\csname LT5\endcsname{\color{black}}%
      \expandafter\def\csname LT6\endcsname{\color{black}}%
      \expandafter\def\csname LT7\endcsname{\color{black}}%
      \expandafter\def\csname LT8\endcsname{\color{black}}%
    \fi
  \fi
  \setlength{\unitlength}{0.0500bp}%
  \begin{picture}(7200.00,5040.00)%
    \gplgaddtomacro\gplbacktext{%
      \csname LTb\endcsname%
      \put(1440,3527){\makebox(0,0)[l]{\strut{}\fontsize{45}{45}\selectfont$\beta(\nicefrac{1}{2},\nicefrac{2}{10})$}}%
    }%
    \gplgaddtomacro\gplfronttext{%
    }%
    \gplbacktext
    \put(0,0){\includegraphics{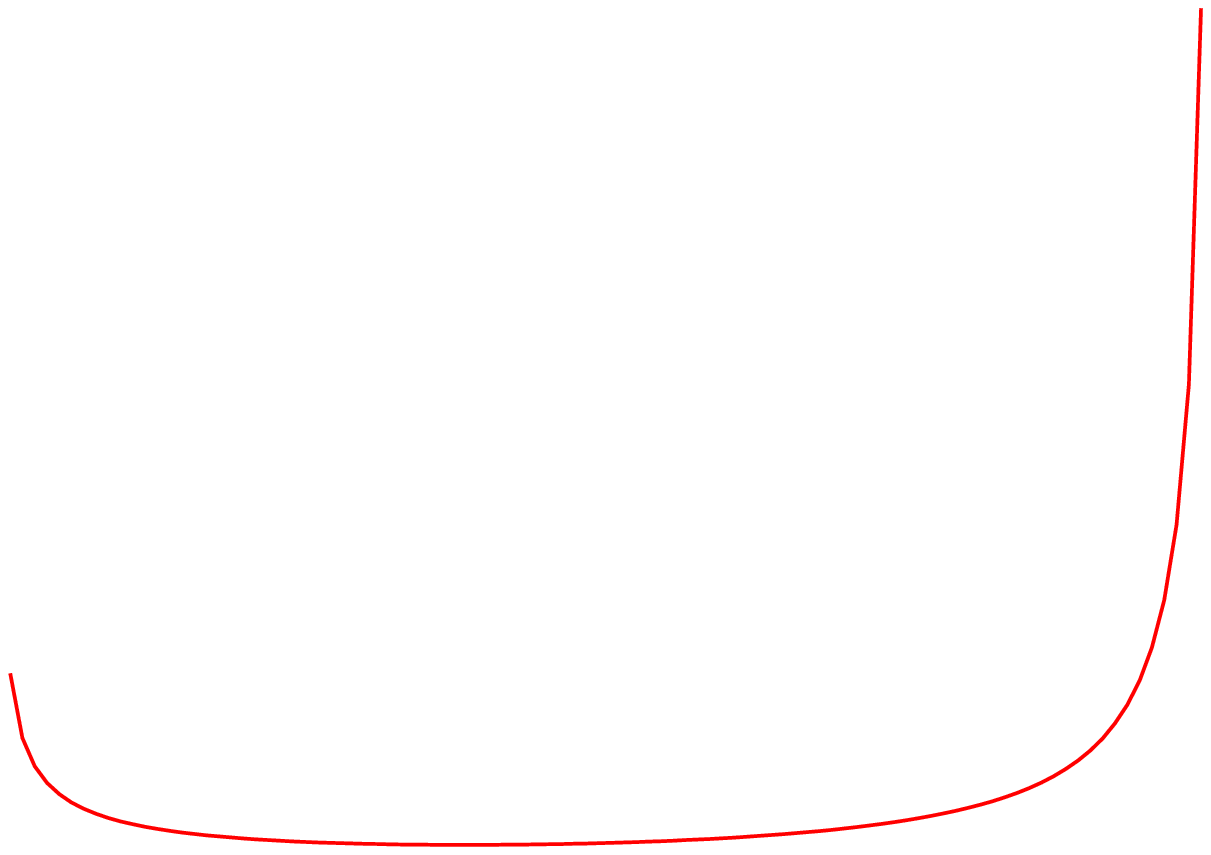}}%
    \gplfronttext
  \end{picture}%
\endgroup
}}
&\multirow{2}{*}{$\rho=0.5$} & \textbf{\numprint{2.7}\%} & \numprint{4.4}\% &
\numprint{7.5}\% &  \numprint{23}\\ 
& & \numprint{6.77}& \numprint{-5.14} & \numprint{-4.11} &
\numprint{0}\\ \cmidrule{2-6}
&\multirow{2}{*}{$\rho=-0.7$} & \textbf{\numprint{4.1}\%} & \numprint{16.9}\% &
\numprint{20.8}\% & \numprint{18.21}\\
& & \numprint{5.74} & \numprint{-8.8} & \numprint{-5.16} &
\numprint{0}\\\bottomrule

\end{tabular}
\caption{Performances of various predictors in the case of Figure~\ref{fig:simple_PW}.
The first line is mean $L^1$ error in \% of deviation from the optimal performance. The second
one is its bias. Bold values are the best performing choices.\label{tab:PW}}
\end{table}

Let us start with the simple case where the vector $\>X$ is just two random
real-valued variables (Figure~\ref{fig:simple_PW}). We repeat \numprint{100000} times
the decimation experiment. In this case, this experiment is just to observe
either $X_1$ or $X_2$ for a given outcome of the vector ($X_1,X_2$) and to predict
the other variable. In addition to the $L^1$ performance, we compute
the biases of the different estimators $\theta$
\begin{equation}\label{eq:biais}
\E_X\left[\theta(X)-\hat\theta_1(X)\right]. 
\end{equation}
We recall that $\hat \theta_1(X)$ is the optimal predictor w.r.t.\ the $L^1$ distance
i.e.\ the conditional median.

The results, for various values of $a$, $b$ and $\rho \egaldef \cov(Y_i,Y_j)$, are
given in Table~\ref{tab:PW}. The first line is the $L^1$
performance \eqref{eq:L1perf} and the second one the bias \eqref{eq:biais}.
Generally, with weak correlations all estimators have a satisfactory behavior.
However the best choice is the cdf function $\Lambda =F $ with the inverse mapping
$\GammaD = F^{-1}$. As expected, the conservative property of the decoding
$\GammaP$ is a real drawback in the case of strong correlation because it
prevents from predicting extreme values (see Figure~\ref{fig:Decod}). When
considering symmetric variables $X_i$, all estimators biases are close to $0$.
Even with $\beta(2,3)$ variables, this bias is negligible. In the case of
asymmetrical variables, these biases are clearly non zero, but do not prevent from
obtaining good performance.

\paragraph{Regular trees.} We consider here the case of a tree with a given
connectivity $n$ for interior nodes: each non leaf node has exactly $n$
neighbors. For $n=3$, we get a binary tree. We perform the decimation experiments and
results are presented in Figure~\ref{fig:Tree}. For the sake of comparison, we show
three other predictors: the median (in red), the $k$ nearest neighbors ($k$-NN,
\citet{Cover-KNN}) predictor (in orange) and the perfect predictor (in black), which
is obtained by computing the conditional mean of the vector $\>Y$. The $k$-NN predictor is
manually optimized to the $k=50$ nearest neighbors in the whole historical data used
to build the model. This predictor is known to be a good choice for road traffic data
(\citet{Smith2002}), but its complexity is too high for large networks
compared to BP\@. Moreover, it requires complete observations of the network, which are
not available when dealing with probe vehicles data.

As a general rule, the choice $(F,F^{-1})$ seems to be the best one.
Let us remark that, if we continue to increase the connectivity $n$, this situation
can change. In fact, at very high connectivity ($n\sim10$), the convergence of mBP
can be very slow. Non convergent cases can then impact the result and one should
rather use $\Lambda_\text{MI}$. Indeed, for the choice $\Lambda_\text{MI}$ the mBP
algorithm is strictly equivalent to BP\@. In this case, the BP algorithm is more stable
since it is always converging on trees. At this point, we discard the choice
($F,\GammaP$) which is clearly inferior to the other ones.
\begin{figure}
\centering
\resizebox{\textwidth}{!}{ \tiny\input{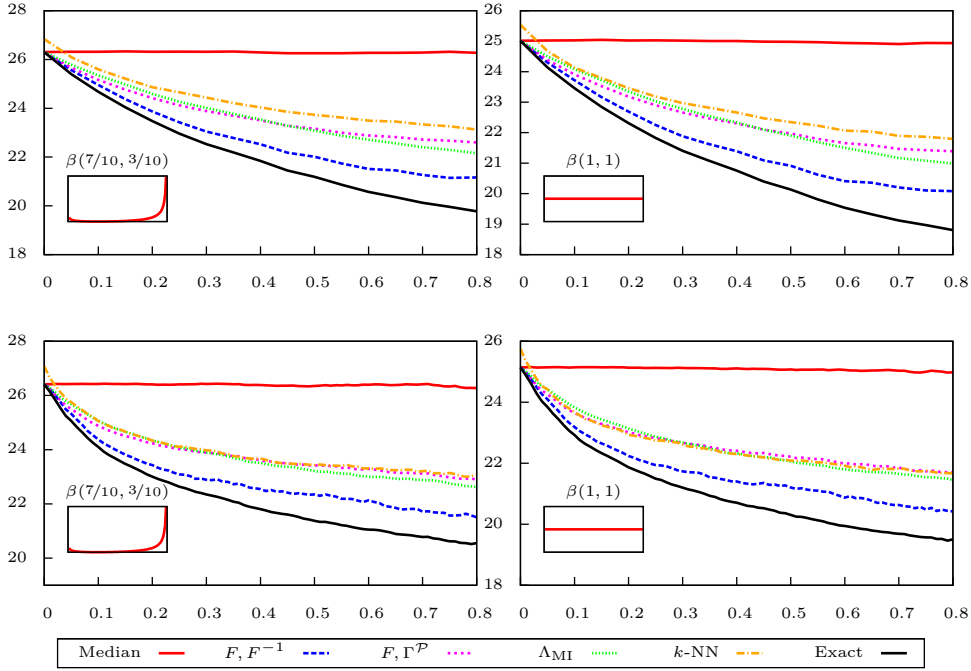}}
\caption{Mean $L^1$ prediction error of unobserved variables ($\times100$) as
a function of the proportion of revealed variables; the small embedded figures are
the corresponding pdf of the bêta variables. The connectivity is $n=3$ for top
figures and $n=5$ for the bottom ones. Each tree contains
\numprint{100} variables.\label{fig:Tree}}
\end{figure}

\begin{figure}
\centering\resizebox{0.5\textwidth}{!}{\begin{tikzpicture}[scale=1]
\begin{scope}

\draw[step=1cm,color=black,style=double,thick] (0 ,0) grid (7,7) ;
\path[-]  (0,0) edge[style=double,bend left,thick] (0,7);
\path[-]  (0,0) edge[style=double,bend right,thick] (7,0);
\path[-]  (7,0) edge[style=double,bend right,thick] (7,7);
\path[-]  (0,7) edge[style=double,bend left,thick] (7,7);

\end{scope} 
  
\end{tikzpicture}}
\caption{A simple model of urban road network with two-way streets.
  The inner grid represents the city itself and the $2\times 4$
  exterior edges form a ring road around it.\label{fig:square++}}
\end{figure}
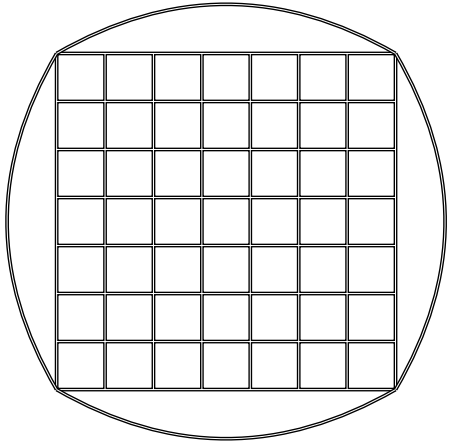

\begin{figure}
\centering
\resizebox{0.8\textwidth}{!}{\input{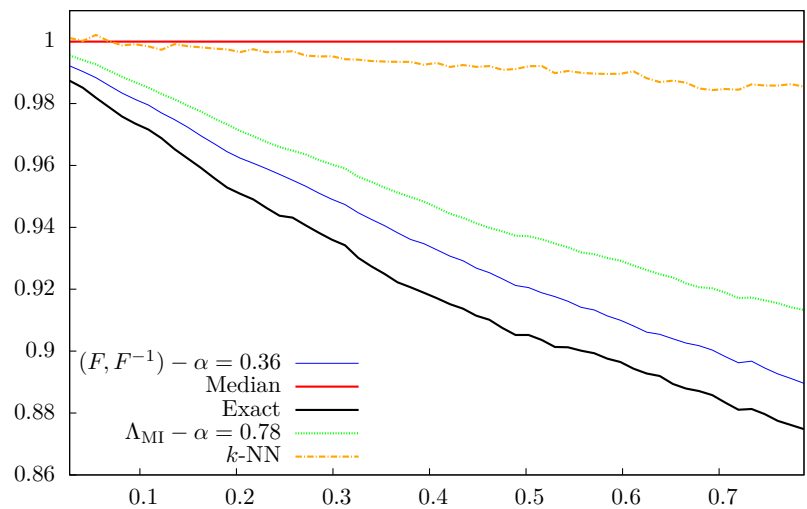}}
\caption{Mean $L^1$ prediction error on
unobserved variables, as a function of the proportion of revealed variables, for the
urban network of Figure~\ref{fig:square++}. All values are relative to
the error made by the ``median'' predictor.
\label{fig:CityModel}}
\end{figure}

\paragraph{A simple road network model.} Let us finally consider a new
synthetic model, associated to the road network of
Figure~\ref{fig:square++}, which is a very rough description of a city
network. The dependency graph of the vector $\>X$ is basically the
line graph of the road networks, i.e.\ there is a direct dependency
between edges $i$ and $j$ iff they are adjacent in the road network.
To model the impact of a ring road on its neighborhood, we set their
partial correlations with adjacent edges to \numprint{0.3}. The
marginal distributions of travel times are real data coming from the
Australian M4 motorway. We assume that the ring road links are always
observed, by means of specific equipment such as magnetic loops.

Again, the decimation experiment is performed \numprint{1000} times
and the results are presented in Figure~\ref{fig:CityModel}. Since the
ring road links are always observed, the decimation curve begin at
$\rho\sim\numprint{0.03}$. Note that, in this case, the k-NN predictor
performance is very bad,  due to the fact that correlations are
small compared to the vector dimension (\citet{Beyer}). The parameter
$\alpha$ of \eqref{eq:alphaBethe} is estimated with a dichotomy search
on $[0,1]$ up to a precision of \numprint{0.01}. Once again, the best
choice of encoding function is the cdf, which performs clearly better
than $\Lambda_\text{MI}$.

\section{Conclusion}\label{sec:ising_concl}

We proposed a simple way to model the interaction between real-valued
random variables defined over a graph from the following information:
\begin{itemize}
 \item the empirical cumulative distribution function of each variable;
 \item an incomplete covariance matrix.
\end{itemize}

The choice of the cdf as encoding function and its inverse as decoding function seems
to be the best one, as long as the graph connectivity is not too high. When this
connectivity increases too much, the algorithm mBP loose its efficiency and one
should rather choose the encoding function $\Lambda_\text{MI}$. An important but
potentially difficult task has been discarded here: finding the dependency graph
structure. This task can be performed using greedy heuristics
(see~\citet{Jalali,FuHaLaMa}) or $L^1$-based regularization method (see~\citet{Ravikumar})

Once the encoding/decoding functions are chosen and the marginals $p_{ij}$
have been estimated, many available methods exists to define the latent
Ising model, i.e. the set of Ising couplings. The best one will depend on the data and determining it will require
tests on real data. However, the results presented here make us quite
optimistic about applying this method to road traffic data, for which the
underlying binary description seems natural. 

Straightforward generalization of the approach presented here can be carried out to
construct latent variables with a feature space larger than $\{0,1\}$, by considering
additional random thresholds defined in Section~\ref{ssec:randthreshold} or deterministic ones; 
the underlying principles remain unchanged.
In particular, it is still possible to build decoding functions based on ML or Bayesian 
updating, to use the EM algorithm for pairwise distributions estimations and the mBP
algorithm for inference.

\bibliography{LatentIsing}
\bibliographystyle{abbrvnat}

\begin{appendix}

\section{Proof of Proposition \ref{prop:stability}}
Let us focus first on the case of one factor with two binary variables
$\sigma_i$ and $\sigma_j$, both observed (Figure~\ref{fig:longpp} with $n=2$).
The  messages
$m_{a \to i}$ are assumed to be normalized such that
\begin{equation*}
  \sum_{s_i} m_{a\to i}(s_i)= 1.
\end{equation*}
We introduce the following notation
\begin{equation*}
 u_n \egaldef m_{a \to i}(0),\quad v_n
\egaldef m_{a \to j}(0),
\end{equation*}
so that $1 - u_n = m_{a \to i}(1)$ and $1 - v_n = m_{a \to j}(1)$.
Using the update rules \eqref{urules}--\eqref{urulesn}, one obtains:
\begin{align}
\label{eq:up_u}
 u_{n+1} &= \frac{\psi_{00}\alpha_j\bar v_n + \psi_{01}\bar\alpha_jv_n}{(\psi_{
00} +
\psi_{10})\alpha_j\bar v_n + (\psi_{01}+\psi_{11} )\bar\alpha_jv_n},\\
\label{eq:up_v}
v_{n+1} &=  \frac{\psi_{00}\alpha_i\bar u_n + \psi_{10}\bar\alpha_iu_n}{(\psi_{
00} +
\psi_{01})\alpha_i\bar u_n + (\psi_{10}+\psi_{11} )\bar\alpha_iu_n}, 
\end{align}
where $\alpha_i \egaldef b^*_i(0)$, $\psi_{yz} \egaldef \psi(\sigma_i = y, 
\sigma_j =z)$ and using the convention $\bar z \egaldef 1 -z$.

\begin{lemma}
\label{lem:u_conv}
The sequences $(u_n)_{n \in \mathbb{N}}$ and $(v_n)_{n \in \mathbb{N}}$, defined
recursively by~\eqref{eq:up_u} and~\eqref{eq:up_v}, converge to a unique fixed 
point
for any $(u_0,v_0) \in ]0,1[^2$. 
\end{lemma}
\begin{proof}
Since the roles of $u_n$ or $v_n$ are symmetric, we will only prove the convergence
of $u_n$. From~\eqref{eq:up_u} and~\eqref{eq:up_v}, we obtain a recursive 
equation of the form $u_{n+2} =f(u_n)$ such as
\begin{equation*}
 f(x) = \frac{h_0 x+K_0}{(h_0+h_1)x + (K_0+K_1)}, 
\end{equation*}
with
\begin{align*}
 h_0 &\egaldef \psi_{00}\alpha_j(\bar \alpha_i \psi_{11} - \alpha_i\psi_{01}) +
\psi_{01}\bar\alpha_j(\psi_{10}\bar\alpha_i - \psi_{00}\alpha_i),\\
 h_1 &\egaldef \psi_{10}\alpha_j(\bar \alpha_i \psi_{11} - \alpha_i\psi_{01}) +
\psi_{11}\bar\alpha_j(\psi_{10}\bar\alpha_i - \psi_{00}\alpha_i),\\
K_0 &\egaldef \psi_{00}\psi_{01}\alpha_i, \quad K_1 \egaldef \alpha_i (
\psi_{10}\psi_{01}\alpha_j + \psi_{11}\psi_{00}\bar \alpha_j).
\end{align*}
The derivative of $f$ is
\begin{equation*} f'(x) =  \frac{h_0 K_1 - h_1 K_0}{\bigl( (h_0+h_1)x + (K_0+K_1)
\bigr)^2},\end{equation*}
which is of constant sign. If $f'(x) \geq 0$, then $u_{2n}$ and
$u_{2n+1}$ are monotonic, and, because $u_n$ is bounded, we can
conclude that both $u_{2n}$ and $u_{2n+1}$ converge. If we could prove
that there is a unique fixed point in the interval $[0,1]$, we would
have proved that $u_n$ converges.

Let us begin by discarding some trivial cases. First the case $f(1) =
1$ implies that $\alpha_i = 1$ and
\begin{align*}
 h_0 &= -\psi_{00}\psi_{01} = -K_0, \\
 h_1 &= -\psi_{10}\psi_{01}\alpha_j - \psi_{11}\psi_{00}\bar\alpha_j =
 -K_1,
\end{align*}
which leads to $f$ being a constant function equal to
$\frac{K_0}{K_0+K_1}$. When $f(0) = 0$, one has $\alpha_i = K_0 = K_1
= 0$, and $f$ is again constant. The cases $f(1)=0$ and $f(0)=1$ are
treated similarly and $f$ is still a constant function, which implies
the trivial convergence of $u_n$.

\paragraph{Case 1: $f$ is increasing}
At least one fixed point exists in $[0,1]$ since $f([0,1])\!\subset [0,1]$.
Studying the roots of $f(x)-x$ shows that the number of fixed points is at most 
$2$ since these fixed points are roots of a degree $2$ polynomial.

Since $f(0)>0$, $f$ being increasing and $f(1)<1$ the number of fixed points 
has to be odd, indeed the graph of $f$ must cross an odd number of times the first 
bisector. One can conclude that there is only one fixed point in $[0,1]$, so both
$u_{2n}$ and $u_{2n+1}$ converge to the same fixed point.

\paragraph{Case 2: $f$ is decreasing}
We just have to consider the sequence $(1-u_n)_{n \in \mathbb{N}}$, which is 
similar, but will be defined by recurrence of the form $1-u_{n+2} = g(1-u_n)$ with a 
function $g$ such as $g'$ is positive and the result of Case~1 applies.
\end{proof}

The case we just studied is in fact much more general than it looks. Indeed, as soon
as a tree gets stuck between exactly two nodes with fixed beliefs, the situation is
equivalent and leads to the result of Proposition \ref{prop:stability}.

\begin{proof}[Proof of Proposition \ref{prop:stability}]
First it is trivial to see that fixing the beliefs of a set of nodes
$\V^* \subset \V$ has the effect of the graph cutting
$\mathcal{T}(\cdot,\V^*)$ in term of messages propagation. To conclude
the proof, it is enough to focus on proving the convergence on a tree with two
leaves in $\V^*$. Consider the tree of Figure~\ref{fig:longpp}; one
can show that it is equivalent to the case of Lemma~\ref{lem:u_conv} for a
well chosen function $\psi$.
\begin{figure}
\setlength{\unitlength}{1mm}
\begin{picture}(50, 15)(-8,-5)
  \put(0,0){\circle*{2}} 
  \put(-1,-4){$\sigma_1$}
  \put(1,0){\line(1,0){8}} 
  \put(9,-2.5){\framebox(10,5){$\psi_{a_1}$}}
  \put(19,0){\line(1,0){8}}
  \put(28,0){\circle{2}} 
  \put(27,-4){$\sigma_2$}
  \put(29,0){\line(1,0){8}}
  \put(37,-2.5){\framebox(10,5){$\psi_{a_2}$}}
  \multiput(48,0)(2,0){4}{\line(1,0){1}} 
  \put(56,-2.5){\framebox(10,5){$\psi_{\!a_{N\!-\!2}}$}}
  \put(66,0){\line(1,0){8}}
  \put(75,0){\circle{2}}
  \put(74,-4){$\sigma_{N\!-\!1}$}
  \put(76,0){\line(1,0){8}}
  \put(84,-2.5){\framebox(10,5){$\psi_{a_{N\!-\!1}}$}}
  \put(94,0){\line(1,0){8}}
  \put(103,0){\circle*{2}}
  \put(102,-4){$\sigma_N$}
\end{picture}
\caption{Chain of $N$ pairwise factors, the extremal variables
$\sigma_1$ and $\sigma_N$ are observed.}
\label{fig:longpp}
\end{figure}
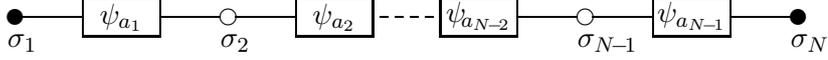
Propagating the updates rules yields $m_{a_1 \to 1}(s_1) \leftarrow \Theta$, 
with 
\begin{equation*}
\Theta \propto \sum_{s_N} \Bigl(\sum_{s_1\ldots s_{N\!-\!2}}
\prod_{i=1}^{N-2}\psi_{a_i}(\mathbf{s}_{a_i}
)\phi_i(s_i)\Bigr)\frac{\psi_{a_{N\!-\!1}}(\mathbf{s}_{a_{N\!-\!1}})
b^*_r(s_N)}{m_{a_{N\!-\!1} \to
N}(s_N)}. 
\end{equation*}
We define $\tilde \psi$ such as
$$\tilde \psi(s_1,s_N) =
\Bigl(\sum_{s_1\ldots s_{N\!-\!2}}\prod_{i=1}^{N-2}\psi_{a_i}(\mathbf{s}_{a_i}
)\phi_i(s_i)\Bigr)\  \psi_{a_{N\!-\!1}}(\mathbf{s}_{a_{N\!-\!1}}),$$
then we use the results of Lemma~\ref{lem:u_conv} to obtain the convergence of
messages on this tree. In the general case of a tree with two leaves in $\V^*$, 
leaves fixed on variables $\sigma_i, i \in \{1\ldots N\}$ will simply 
affect the local fields $\phi_i$. The leaves fixed on factors $a_i$ will affect the
functions $\psi_{a_i}$. In fact, since the graph is a tree, we know that the
information sent by $\sigma_1$ and $\sigma_N$ to these leaves will not come 
back to $\sigma_1$ and $\sigma_N$. These leaves send constant messages, which can be
integrated into the functions $\psi$ and $\phi$, in order to recover the setting
 of Lemma~\ref{lem:u_conv}.
\end{proof}

\end{appendix}

\end{document}